\documentclass[11pt,oneside]{amsart}

\usepackage{setspace}
\usepackage[
paper=a4paper,
text={138mm,195mm},centering
]{geometry}

\usepackage{fullpage}
\usepackage{amssymb,amsxtra}
\usepackage{mathrsfs}
\usepackage{float, placeins}
\usepackage[all]{xy}
\usepackage{pb-diagram,pb-xy}
\usepackage{graphicx,color,float}
\usepackage[bookmarks]{hyperref}
\usepackage{pinlabel}
\usepackage[
  textwidth=1.2in,
  backgroundcolor=yellow,
  bordercolor=orange,
  textsize=small
]{todonotes}
\iftrue
\makeatletter
\def\@settitle{%
  \vspace*{-20pt}
  \begin{flushleft}%
    \baselineskip14\p@\relax
    \normalfont\bfseries\LARGE
    \@title
  \end{flushleft}%
}
\def\@setauthors{%
  \begingroup
  \def\thanks{\protect\thanks@warning}%
  \trivlist
  \raggedright
  \large \@topsep30\p@\relax
  \advance\@topsep by -\baselineskip
  \item\relax
  \author@andify\authors
  \def\\{\protect\linebreak}%
  \authors
  \ifx\@empty\contribs
  \else
    ,\penalty-3 \space \@setcontribs
    \@closetoccontribs
  \fi
  \normalfont
  \@setaddresses
  \endtrivlist
  \endgroup
}
\def\@setaddresses{\par
  \nobreak \begingroup
  \small
  \def\author##1{\nobreak\addvspace\smallskipamount}%
  \def\\{\unskip, \ignorespaces}%
  \interlinepenalty\@M
  \def\address##1##2{\begingroup
    \par\addvspace\bigskipamount\noindent
    \@ifnotempty{##1}{(\ignorespaces##1\unskip) }%
    {\ignorespaces##2}\par\endgroup}%
  \def\curraddr##1##2{\begingroup
    \@ifnotempty{##2}{\nobreak\noindent\curraddrname
      \@ifnotempty{##1}{, \ignorespaces##1\unskip}\/:\space
      ##2\par}\endgroup}%
  \def\email##1##2{\begingroup
    \@ifnotempty{##2}{\nobreak\noindent E-mail address%
      \@ifnotempty{##1}{, \ignorespaces##1\unskip}\/:\space
      \ttfamily##2\par}\endgroup}%
  \def\urladdr##1##2{\begingroup
    \def~{\char`\~}%
    \@ifnotempty{##2}{\nobreak\noindent\urladdrname
      \@ifnotempty{##1}{, \ignorespaces##1\unskip}\/:\space
      \ttfamily##2\par}\endgroup}%
  \addresses
  \endgroup
  \global\let\addresses=\@empty
}
\def\@setabstracta{%
    \ifvoid\abstractbox
  \else
    \skip@25\p@ \advance\skip@-\lastskip
    \advance\skip@-\baselineskip \vskip\skip@
    \box\abstractbox
    \prevdepth\z@ 
    \vskip-10pt
  \fi
}
\renewenvironment{abstract}{%
  \ifx\maketitle\relax
    \ClassWarning{\@classname}{Abstract should precede
      \protect\maketitle\space in AMS document classes; reported}%
  \fi
  \global\setbox\abstractbox=\vtop \bgroup
    \normalfont\small
    \list{}{\labelwidth\z@
      \leftmargin0pc \rightmargin\leftmargin
      \listparindent\normalparindent \itemindent\z@
      \parsep\z@ \@plus\p@
      
    }%
    \item[\hskip\labelsep\bfseries\abstractname.]%
}{%
  \endlist\egroup
  \ifx\@setabstract\relax \@setabstracta \fi
}
\def\section{\@startsection{section}{1}%
  \z@{-1.2\linespacing\@plus-.5\linespacing}{.8\linespacing}%
  {\normalfont\bfseries\Large}}
\def\subsection{\@startsection{subsection}{2}%
  \z@{-.8\linespacing\@plus-.3\linespacing}{.3\linespacing\@plus.2\linespacing}%
  {\normalfont\bfseries}}
\def\subsubsection{\@startsection{subsection}{3}%
  \z@{.7\linespacing\@plus.2\linespacing}{-1.5ex}%
  {\normalfont\itshape}}
\def\@secnumfont{\bfseries}
\makeatother
\fi 

\def\to{\mathchoice{\longrightarrow}{\rightarrow}{\rightarrow}{\rightarrow}}
\makeatletter
\newcommand{\shortxra}[2][]{\ext@arrow 0359\rightarrowfill@{#1}{#2}}
\def\longrightarrowfill@{\arrowfill@\relbar\relbar\longrightarrow}
\newcommand{\longxra}[2][]{\ext@arrow 0359\longrightarrowfill@{#1}{#2}}

\makeatother

\makeatletter
\def\Nopagebreak{\@nobreaktrue\nopagebreak}
\makeatother

\theoremstyle{plain}
\newtheorem{theorem}{Theorem}[section]
\newtheorem{proposition}[theorem]{Proposition}
\newtheorem{corollary}[theorem]{Corollary}
\newtheorem{lemma}[theorem]{Lemma}

\newtheorem{maintheorem}{Theorem}

\theoremstyle{definition}
\newtheorem{definition}[theorem]{Definition}

\newtheorem{example}[theorem]{Example}
\newtheorem*{remark}{Remark}

\def\Z{\mathbb{Z}}
\def\Q{\mathbb{Q}}
\def\R{\mathcal{R}}
\def\C{\mathcal{C}}

\def\vp{\varphi}
\def\F{\mathcal{F}}

\def\P{\mathcal{P}}
\def\a{\alpha}
\def\b{\beta}

\def\T{\mathcal{T}}

\def\ba{\begin{array}}
\def\ea{\end{array}}

\let\oldsharp=\# \def\#{\mathbin{\oldsharp}}

\def\Bl{B\ell}
\def\lk{\operatorname{lk}}
\def\Arf{\operatorname{Arf}}

\def\wt{\widetilde}

\def\F{\mathcal{F}}
\def\G{\mathcal{G}}

\def\Z{\mathbb{Z}}
\def\Q{\mathbb{Q}}
\def\S{\mathcal{S}}

\def\d{\partial}

\def\ol{\overline}

\def\tilde{\widetilde}

\def\lk{\operatorname{lk}}

\def\+{\oplus}

\def\amalgover#1#2{\mathop{\amalg}_{#1}^{#2}}

\def\#{\mathbin{\oldsharp}}

\def\emptystr{}
\newcommand{\mkc}[2][]{\begin{color}{red}#2%
  \def\tempstr{#1}%
  \ifx\tempstr\emptystr \else\textsf{\SMALL\ \raise.7ex\hbox{[\tempstr]}}\fi
\end{color}}

\begin{document}
\title{Two-torsion in the grope and solvable filtrations of knots}
\author{Hye Jin Jang}
\address{Department of Mathematics, POSTECH, Pohang, 790--784, Republic of Korea}
\email{hyejin.jang1986@gmail.com}

\def\subjclassname{\textup{2010} Mathematics Subject Classification}
\expandafter\let\csname subjclassname@1991\endcsname=\subjclassname
\expandafter\let\csname subjclassname@2000\endcsname=\subjclassname

\begin{abstract}
We study knots of order $2$ in the grope filtration $\{\G_h\}$ and the solvable filtration $\{\F_h\}$ of the knot concordance group.
We show that, for any integer $n\ge4$, there are knots generating a $\Z_2^\infty$ subgroup of $\G_n/\G_{n.5}$.
Considering the solvable filtration, our knots generate a $\Z_2^\infty$ subgroup of $\F_n/\F_{n.5}$ $(n\ge2)$ distinct from the subgroup generated by the previously known 2-torsion knots of Cochran, Harvey, and Leidy.
We also present a result on the 2-torsion part in the Cochran, Harvey, and Leidy's primary decomposition of the solvable filtration. 
\end{abstract}

\maketitle

\section{Introduction}
Two oriented knots $K$ and $J$ in $S^3$ are said to be \emph{(topologically) concordant} if $K\times\{0\}$ and $-J\times \{1\}$ cobound a locally flat annulus in $S^3\times[0,1]$, where $-J$ denotes the mirror image of $J$ with orientation reversed.
Concordance is an equivalence relation on the set of all oriented knots in $S^3$, and concordance classes form an abelian group under connected sum.
It is called the \emph{knot concordance group}.
A knot which represents the identity is called a \emph{slice knot}.
Understanding the structure of the knot concordance group has been one of the main interests in the study of knot theory.

In the celebrated paper \cite{cochran2003knot} of Cochran, Orr, and Teichner, they introduced several types of approximations of a knot being slice.
They defined the concept of \emph{knots bounding (symmetric) gropes of height $h$ in $D^4$} for each half integer $h\geq1$, relaxing the geometric condition of a knot bounding a slice disk.
Also, a knot whose zero framed surgery bounds a $4$-manifold which resembles slice knot exteriors is said to be $(h)$-solvable, where the half integer $h\ge0$ depends on the degree of resemblance.
(See Section~7 and 8 of \cite{cochran2003knot} for the precise definitions of a grope and $(h)$-solvability.)
Any slice knot bounds a grope of arbitrary height in $D^4$ and is $(h)$-solvable for any $h$.
Also, a knot bounding a grope of height $h+2$ is $(h)$-solvable \cite[Theorem~8.11]{cochran2003knot}, while the converse is unknown.

The set of concordance classes of knots bounding gropes of height $h$ in $D^4$ is a subgroup of $\C$, which is denoted by $\G_h$.
Similarly, the set of concordance classes of $(h)$-solvable knots is a subgroup of $\C$, which is denoted by $\F_h$.
They form filtrations in the knot concordance group~$\C$:
\[\{0\}\leq\cdots\leq\G_{h+0.5}\leq\G_h\leq\cdots,\G_{1.5}\leq\G_{1}\leq\C,\]
\[\{0\}\leq\cdots\leq\F_{h+0.5}\leq\F_h\leq\cdots,\F_{0.5}\leq\F_0\leq\C,\]
which are called the \emph{grope} and \emph{solvable filtration} respectively.

Recently, to understand the structure of the knot concordance group, the graded quotients $\{\G_n/\G_{n.5}\mid n\in\Z_{\ge1}\}$ and $\{\F_n/\F_{n.5}\mid n\in\Z_{\ge0}\}$ have been studied extensively and successfully.
(See \cite{cochran2003knot,cochran2004structure,  cochran2007knot,cha2007the,cochran2008higher-order, kim2008polynomial,cochran2009knot, horn2010the, cochran20112-torsion,cochran2011primary, cha2012l2, cha2010amenable}, for example.)
Interestingly, nothing is known about the structure of $\G_{n.5}/\G_{n+1}$ or $\F_{n.5}/\F_{n+1}$ except $\G_{1.5}/\G_2=0$.

In this paper, we study order 2 elements in the knot concordance group and those filtrations.
There were some results on the order 2 elements in the solvable filtration.
For instance, $\C/\F_0$ is isomorphic to $\Z_2$ with Arf invariant map as an isomorphism.
Also, the study of J.~Levine on the algebraic concordance group \cite{levine1970an} implies that $\F_0/\F_{0.5}$ contains $\Z^\infty_2$ as a subgroup.
In \cite{livingston1999order}, C.~Livingston found infinitely many order 2 elements in $\C$ using Casson-Gordon invariants~\cite{casson1978on,casson1986cobordism}. 
Using the fact that the Casson-Gordon invariants vanish for $(1.5)$-solvable knots \cite[Section~9]{cochran2003knot}, Livingston's method also gives infinitely many examples in $\F_1$, which generate a $\Z_2^\infty$ subgroup of $\F_1/\F_{1.5}$.
In \cite{cochran20112}, T.~Cochran, S.~Harvey, and C.~Leidy showed that for each integer $n\geq2$, there are infinitely many 2-torsion elements in $\F_n$ which generate a $\Z_2^\infty$ subgroup of $\F_n/\F_{n.5}$.

Regarding the grope filtration, much less was known about 2-torsion.
It is known that $\G_{1.5}$ and $\G_2$ are isomorphic to $\F_0$ and $\G_{2.5}$ is isomorphic to $\F_{0.5}$ \cite[Theorem~5]{teichner2002knots}.
Hence the results on the solvable filtration imply that $\G_1/\G_{1.5}$ is isomorphic to $\Z_2$ and $\G_2/\G_{2.5}$ contains $\Z_2^\infty$ as a subgroup.

In this paper, we show that $\G_n/\G_{n.5}$ has infinitely many 2-torsion elements for any $n\geq4$:

\begin{maintheorem}\label{theorem:maintheorem1}
For any integer $n\geq2$, there is a subgroup of $\G_{n+2}/\G_{n+2.5}$ isomorphic to $\Z_2^\infty$.
Moreover, the subgroup is generated by $(n)$-solvable knots with vanishing poly-torsion-free-abelian (abbreviated PTFA) $L^2$-signature obstructions.
\end{maintheorem}

Discussions on (PTFA) $L^2$-signature obstructions and vanishing PTFA $L^2$-signature obstructions property appear at the end of this section and Section~\ref{section:preliminaries}.

Note that there is another concept of \emph{knots bounding (symmetric) gropes of height $h$ in $D^4$}, which can be viewed as an approximation of having slice disks (see Section~7 of \cite{cochran2003knot}).
Similarly to the grope case, the sets of concordance classes of knots bounding Whitney towers of height $h$, for every half integer $h\ge1$, forms a filtration on $\C$.
It is known that a knot bounding a grope of height $h$ also bounds a Whitney tower of the same height (see \cite[Corollary~2]{schneiderman2006whitney}).
Hence our examples generate $\Z_2^\infty$ subgroups in the successive quotients of Whitney tower filtration as~well.

To construct knots in Theorem~\ref{theorem:maintheorem1}, we adopt the iterated infection construction in a similar way to the Cochran, Harvey, and Leidy's method in \cite{cochran20112} but with different infection knots.
Using the infinite set of knots bounding gropes of height 2 in $D^4$ constructed by P. Horn \cite{horn2010the} and the infection knots used in \cite{cha2010amenable} as our basic ingredients, we can find infection knots with desired properties.

Figure~\ref{figure:example1} illustrates an order 2 knot bounding a grope of height 4, but not height 4.5, in $D^4$ (not necessarily with vanishing PTFA $L^2$-signature obstructions).
Note that it is constructed by the iterated satellite construction ${K^k}_{\eta_k}(K_{\a,\b}(J,\ol{J}))$, where $K$ and $K^k$ are as in Figure~\ref{figure:K} and \ref{figure:K^k} and $J$ is the Cochran-Teichner's knot bounding height 2 (see \cite[Figure~3.6]{cochran2007knot}).
The detailed explanation on the construction will be given later.

\begin{figure}[htbp]
   \labellist
   \hair 0mm
   \pinlabel{$2788531200$ copies} at 237 302
   \pinlabel{$2788531200$ copies} at 237 262
   \pinlabel{$\dots$} at 237 548
   \pinlabel{$\dots$} at 237 440
   \pinlabel{$\dots$} at 237 130
   \pinlabel{$\dots$} at 237 14
   \endlabellist
   \centering
   \includegraphics{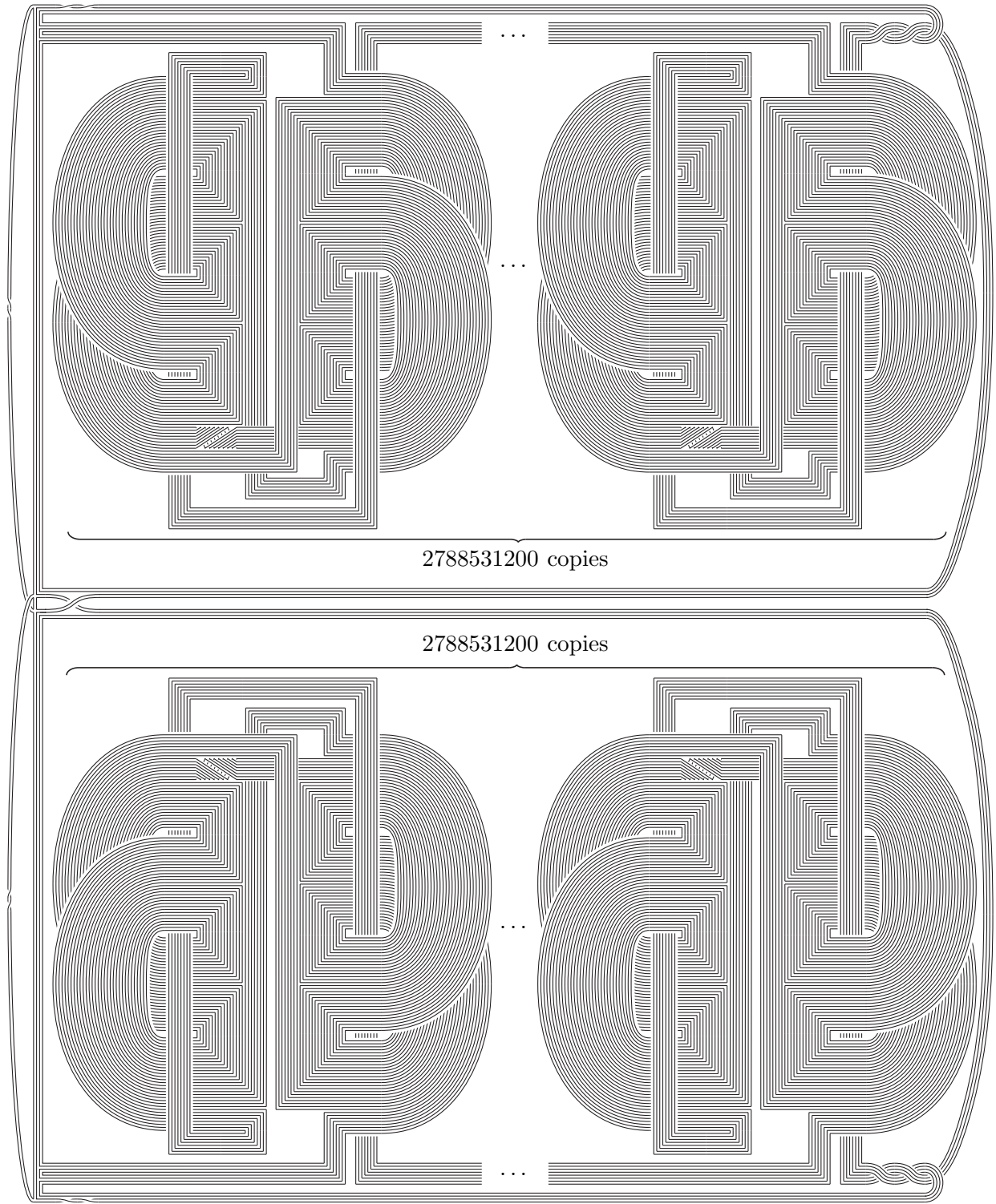}
   \caption{An example of order 2 element in $\G_4/\G_{4.5}$}
   \label{figure:example1}
\end{figure}

In the proof of Theorem~\ref{theorem:maintheorem1}, to show certain knots are not in $\G_{n.5}$ we actually prove that they are not $(n.5)$-solvable.
That $\F_h$ is contained in $\G_h$ implies that the knots in Theorem~\ref{theorem:maintheorem1} generating a $\Z_2^\infty$ subgroup of $\G_n/\G_{n.5}$ also generate a $\Z_2^\infty$ subgroup of $\F_n/\F_{n.5}$.

Recall that Cochran, Harvey, and Leidy found knots which generate a $\Z_2^\infty$ subgroup of $\F_n/\F_{n.5}$.
We prove that our knots are distinct from Cochran, Harvey, and Leidy's knots in $\F_n/\F_{n.5}$:

\begin{maintheorem}\label{theorem:maintheorem3}
For any integer $n\geq2$, there are $(n)$-solvable knots with vanishing PTFA $L^2$-signature obstructions which generate a $\Z_2^\infty$ subgroup of $\F_n/\F_{n.5}$.
This subgroup has trivial intersection with the subgroup generated by Cochran, Harvey and  Leidy's knots in~\cite{cochran20112}.
\end{maintheorem}

Cochran, Harvey, and Leidy's knots do not have the property of vanishing PTFA $L^2$-signature obstructions.
(Recall that, in\cite{cochran20112}, Cochran, Harvey, and Leidy's knots have been shown not to be  $(n.5)$-solvable by calculating a nonzero PTFA $L^2$-signature obstruction.)
The property of PTFA $L^2$-signature obstructions of our knots is a distinct feature of our knot from Cochran, Harvey, and Leidy's knots and this implies the second part of Theorem~\ref{theorem:maintheorem3}.

\subsubsection*{Results on the primary decomposition structures}

In \cite{cochran2011primary} they defined the concept of \emph{$(h,\P)$-solvability of knots} which is a generalization of $(h)$-solvability.
Here $h$ is a nonnegative half integer and $\P=(p_1(t),p_2(t),\dots, p_{\lfloor h\rfloor}(t))$ is an $\lfloor h\rfloor$-tuple of Laurent polynomials $p_i(t)$ over $\Q$.
They showed that the subgroups $\F_h^\P\le\C$ of $(h,\P)$-solvable knots form another filtration on $\C$ which is coarser than $\{\F_h\}$, that is, $\F_h\leq\F^\P_h$.

Hence the quotient maps induce
\[\frac{\F_n}{\F_{n.5}}\,\to\,\prod_\P\,\,\frac{\F_n}{\F_{n.5}^\mathcal{P}\cap\F_n},\]
where the product is taken over all $n$-tuples $\P$ as above.
This is called the \emph{primary decomposition structures} of the knot concordance group by Cochran, Harvey, and Leidy~\cite{cochran2011primary}.

They used the primary decomposition structures crucially in their proof of the existence of $\Z_2^\infty$ subgroup of $\F_n/\F_{n.5}$ as follows.
For each $\P$, with some additional conditions, they show that there is a negatively amphichiral $(n)$-solvable knot which does not vanish in ${\F_n}/({\F_{n.5}^\P\cap\F_n})$ but vanishes in ${\F_n}/({\F_{n.5}^\mathcal{Q}\cap\F_n})$ for any $n$-tuple $\mathcal{Q}$ which is \emph{strongly coprime to $\P$.}
(For the definition of strong coprimeness and precise conditions required for $\P$, see Section~\ref{section:hp-solvability} or \cite[Definition~4.4]{cochran2011primary}.)
By using an infinite set of $\P$ which are pairwise strongly coprime, they showed the existence of $\Z_2^\infty$ subgroup of $\F_n/\F_{n.5}$.

On the other hand, our proof of Theorems~\ref{theorem:maintheorem1} and \ref{theorem:maintheorem3} does not involve the use of the primary decomposition structures.
This might be regarded as a simpler proof to the existence of a $\Z_2^\infty$ subgroup of $\F_n/\F_{n.5}$.
We are able to do that by using the recently introduced \emph{amenable} $L^2$-signature obstructions in the study of the knot concordance by Cha and Orr which we explain later.

Using our method, we also find new 2-torsion elements in the primary decomposition structures.
A corollary of Cochran, Harvey, and Leidy's result \cite[Theorem~5.3 and Threorem~5.5]{cochran20112-torsion} is that for any $n\ge2$ there is a set $S$ of infinitely many $n$-tuples $\P$ of the form $(0,p_2(t),\dots,p_n(t))$ such that, for any $\P\in S$, there are infinitely many $(n)$-solvable knots of order $2$ which
\begin{enumerate}
\item generate a $\Z_2^\infty$ subgroup in  $\F_n/(\F_{n.5}^\P\cap\F_n)$, but
\item vanish in $\F_n/(\F_{n.5}^\mathcal{Q}\cap\F_n)$, for any $\mathcal{Q}\in S$, $\mathcal{Q}\neq\P$.
\end{enumerate}
A precise description of $S$ appears in Section~\ref{section:proof-of-theorem-B}.
In this setting, we show the following:

\begin{maintheorem}\label{theorem:maintheorem2}
For each $\P\in S$, there are infinitely many $(n)$-solvable knots which satisfy above (1) and (2), but generate a $\Z_2^\infty$ subgroup in $\F_n/(\F_{n.5}^\P\cap\F_n)$ having trivial intersection with Cochran, Harvey, and Leidy's subgroup.
\end{maintheorem}

\subsubsection*{The obstructions used: amenable $L^2$-signature defects}
The obstructions we use to detect non-$(n.5)$-solvable, or non-$(n.5,\P)$-solvable knots are amenable von Neumann $\rho$-invariants, and equivalently, $L^2$-signature defects.

It has been known that certain $\rho$-invariants of the zero
framed surgery on an $(n.5)$-solvable knot over PTFA groups vanish (\cite[Theorem~4.2]{cochran2003knot}, Theorem~\ref{thm:homology-cobordism-invariance-on-PTFA}).
They have been used as the key ingredient to detect many non-$(n.5)$-solvable knots.
Cochran, Harvey, and Leidy extended the vanishing theorem of PTFA $\rho$-invariants to $(n.5,\P)$-solvable knots and found their 2-torsion elements.

In 2009, J. Cha and K. Orr \cite{cha2012l2} extended the homology cobordism invariance of $\rho$-invariants to \emph{amenable groups lying in Strebel's class}.
Since then amenable $\rho$-invariants have been used for the study of various problems on homology cobordism and knot concordance.
For example, Cha and Orr \cite{cha2013hidden} found infinitely many hyperbolic 3-manifolds which are not pairwisely homology cobordant but cannot be distinguished by any previously known methods.
Also, Cha, Friedl, and Powell \cite{cha2012symmetric, cha2013non-concordant, cha2012concordance} detected many non-concordant links which have not been detected previously.
Especially, Cha \cite{cha2010amenable} found a subgroup of $\F_n/\F_{n.5}$ isomorphic to $\Z^\infty$, which cannot be detected by any PTFA $\rho$-invariants.

In this paper, we utilize amenable $\rho$-invariants to detect 2-torsion elements.
The use of amenable groups as the image of group homomorphisms enables us to detect various 2-torsion elements described in Theorems~\ref{theorem:maintheorem1}, \ref{theorem:maintheorem3},  and \ref{theorem:maintheorem2}, which are unlikely to be detected by using PTFA $L^2$-signature obstructions.
Also, the fact that there are a lot more independent amenable $\rho$-invariants enables us to show the existence of $\Z_2^\infty$ subgroups in Theorems~\ref{theorem:maintheorem1}, \ref{theorem:maintheorem3},  and \ref{theorem:maintheorem2} without using the primary decomposition structures.
In particular, we use amenable $\rho$-invariant over groups with $p$-torsion for various choices of prime $p$ in the proof.

We also employ the recent result of Cha on Cheeger-Gromov bounds \cite{cha2014a}.
Note that Cochran, Harvey, and Leidy's 2-torsion knots are not fully constructive because the explicit number of connected summands of infection knots needed was unknown.
It was totally due to the absence of an explicit estimate of the universal bound for $\rho$-invariants of a 3-manifold (see Theorem~\ref{thm:cheeger-gromov-bound}).
Recently, Cha found an explicit universal bound in terms of topological descriptions of a 3-manifold (see Theorem~\ref{theorem:explicit-cheeger-gromov}).
Using this result, we can present our example in a fully constructive way.
We remark that not only our knots but also many previously known non-$(n.5)$-solvable knots can be modified to be fully constructive.

\vskip 5mm

This article is organized as follows.
In Section~\ref{section:preliminaries}, we provide definitions and properties of von Neumann $\rho$-invariants of closed 3-manifolds and $L^2$-signature defects of 4-manifolds, which will be used in this paper.
In Section~\ref{section:construction}, we give a specific construction of knots bounding gropes of height $n+2$ in $D^4$ which will be the prototype of 2-torsion knots in Theorems~\ref{theorem:maintheorem1}, \ref{theorem:maintheorem3}, and~\ref{theorem:maintheorem2}.
In Sections~\ref{section:proof-of-theorem-A}  Theorems~\ref{theorem:maintheorem1} and~\ref{theorem:maintheorem3} are proven.
We discuss the definition of the refined filtration related to a tuple of integral Laurent polynomials of Cochran, Harvey, and Leidy in Section~\ref{section:hp-solvability}.
In Section~\ref{section:proof-of-theorem-B}, we prove Theorem~\ref{theorem:maintheorem3}.

\subsection*{Acknowledgements}
The author thanks her advisor Jae Choon Cha for guidance and encouragement on this project.
The author is partially supported by NRF grants 2013067043 and 2013053914.

\section{Preliminaries on von Neumann $\rho$-invariants}\label{section:preliminaries}

The \emph{von Neumann $\rho$-invariants on closed 3-manifolds}, and equivalently, the \emph{$L^2$-signature defects on 4-manifolds} can be used as obstructions for a knot to being $(n.5)$-solvable \cite{cochran2003knot,cha2010amenable}.
In this section, we give a brief introduction to these invariants and their properties which are useful for our purposes. 
For more details, we recommend \cite[Section~5]{cochran2003knot}, \cite[Section~2]{cochran2007knot}, and \cite[Section~3]{cha2010amenable}.

Let $M$ be a closed, smooth, oriented 3-manifold with a Riemannian metric $g$.
Let $G$ be a countable group and $\phi\colon \pi_1(M)\to G$ a group homomorphism.
Let $\eta(M, g)$ be the $\eta$-invariant of the odd signature operator of $(M,g)$.
Cheeger and Gromov defined the von Neumann $\eta$-invariant $\eta^{(2)}(M,\phi,g)$ by lifting the metric $g$ and the signature operator to the $G$-cover of $M$ and using the von Neumann trace.
Then the \emph{von Neumann $\rho$-invariant $\rho(M,\phi)$} is defined as the difference between $\eta(M,g)$ and $\eta^{(2)}(M,\phi,g)$.
It is known that $\rho(M,\phi)$ is a real-valued topological invariant which does not depend on the choice of $g$ \cite{cheeger1985bounds}.

On the other hand, let $W$ be a 4-manifold and $\phi\colon\pi_1(W)\to G$ be a group homomorphism.
Then there is an invariant so-called the \emph{$L^2$-signature defect} of $(W,\phi)$, which is denoted by $S(W,\phi)$ in this article.
It is known that $S(W,\phi)$ is equal to $\rho(\d W,\phi)$ where $\phi$ is the restriction of $\phi\colon W\to G$ on $\d W$ \cite{mathai1992spectral,ramachandran1993von}.
In this article this fact will be used frequently.

The following theorem of Cochran, Orr, and Teichner enables to use $\rho$-invariants as obstructions for a knot being $(n.5)$-solvable:

\begin{theorem}\cite[Theorem~4.2]{cochran2003knot}\label{thm:homology-cobordism-invariance-on-PTFA}
Let $K$ be an $(n.5)$-solvable knot and $M(K)$ be the zero
framed surgery on $K$.
Let $G$ be a PTFA group whose $(n+1)$-th derived subgroup $G^{(n+1)}$ vanishes, and $\phi\colon\pi_1(M(K))\to G$ be a group homomorphism which extends to an $(n.5)$-solution for $K$.
Then $\rho(M(K),\phi)$ vanishes.
\end{theorem}

Note that $G$ is assumed to be a \emph{PTFA} group, that is, it has a subnormal series
\[0=G^m\lhd G^{m-1}\lhd \dots G^1\lhd G^0= G\]
such that every $G^i/G^{i+1}$ is torsion-free abelian.

In \cite{cha2012l2}, Cha and Orr showed the homology cobordism invariance of $\rho$-invariants over \emph{amenable groups lying in Strebel's class $D(R)$} for a ring $R$, which will be called \emph{amenable $\rho$-invariants} in this paper.
Using the result, Cha extended the vanishing property of PTFA $\rho$-invariants in Theorem~\ref{thm:homology-cobordism-invariance-on-PTFA} for $(n.5)$-solvable knots to amenable $\rho$-invariants as follows:

\begin{theorem}\cite[Theorem~1.3]{cha2010amenable}\label{thm:homology-cobordism-invariance-on-amenable}
Let $K$ be an $(n.5)$-solvable knot.
Let $G$ be an amenable group lying in Strebel's class $D(R)$ where $R$ is $\Q$ or $\Z_p$, and $G^{(n+1)}=\{e\}$.
Let $\phi\colon\pi_1(M(K))\to G$ be a group homomorphism which extends to an $(n.5)$-solution for $K$ and sends the meridian of $K$ to an infinite order element in $G$.
Then $\rho(M(K),\phi)$ vanishes.
\end{theorem}

The definition of amenable groups lying in Strebel's class will not be presented in this paper since the following lemma is sufficient for our purposes:

\begin{lemma}\cite[Lemma~6.8]{cha2012l2}\label{lemma:amenable_group}
Suppose $G$ is a group admitting a subnormal series
\[0=G^m\lhd G^{m-1}\lhd \dots G^1\lhd G^0= G\]
whose quotient $G^{i}/G^{i+1}$ is abelian.
Let $p$ be a prime.
If every $G^i/G^{i+1}$ has no torsion coprime to $p$, then $G$ is amenable and in $D(\Z_p)$.
If every $G^i/G^{i+1}$ is torsion-free, then $G$ is amenable and in $D(R)$ for any ring~$R$.
\end{lemma}

As seen in this lemma, the class of amenable groups in Strebel's class $D(R)$ is much larger than and subsumes PTFA groups as a special case.

A natural question here is whether there is a knot whose non-$(n.5)$-solvability cannot be detected by PTFA $\rho$-invariants but can be detected by amenable $\rho$-invariants. In this context, Cha introduced the concept of \emph{$(n)$-solvable knots with vanishing PTFA $L^2$-signature obstructions}:

\begin{definition}\cite[Definition~4.7]{cha2010amenable}\label{def:vanishing-l2-signature-obstruction}
A knot $K$ is said to be \emph{$(n)$-solvable with vanishing PTFA $L^2$-signature obstructions} if there is an $(n)$-solution $W$ for $K$ such that for any PTFA group $G$ and for any group homomorphism $\phi\colon \pi_1(M(K))\to G$ which extends to $\pi_1(W)$, $\rho(M(K),\phi)$ vanishes.
\end{definition}

It turns out that the set $\mathcal{V}_n\le\C$ of classes of $(n)$-solvable knots with vanishing PTFA $L^2$-signature obstructions is a subgroup between $\F_n$ and $\F_{n.5}$ (see \cite[Proposition~4.8]{cha2010amenable}).
Cha showed that there are infinitely many classes of $(n)$-solvable knots with vanishing PTFA $L^2$-signature obstructions which are linearly independent in $\F_n/\F_{n.5}$. In other words, there is a subgroup of $\mathcal{V}_n/\F_{n.5}$ which is isomorphic to $\Z^\infty$ (Theorem~1.4 of \cite{cha2010amenable}).
Note that all of the previously known $(n)$-solvable knots which are not $(n.5)$-solvable (\cite{cochran2003knot, cochran2004structure, cochran2007knot,cochran2009knot, horn2010the, cochran2011primary, cochran20112-torsion}) are not in $\mathcal{V}_n$, since the proofs of their non-$(n.5)$-solvability are actually done by showing that they are not $(n)$-solvable with vanishing PTFA $L^2$-signature obstructions.
Hence Cha's knots are distinguished from any previously known nontrivial elements in $\F_n/\F_{n.5}$.


We close this section with some useful properties about $\rho$-invariants and $L^2$-signature defects for our purposes.

\begin{theorem}\cite{cheeger1985bounds, cha2014a}\label{thm:cheeger-gromov-bound}
For any closed 3-manifold $M$, there is a constant $C_M>0$ such that, for any homomorphism $\phi\colon \pi_1(M)\to G$, $|\rho(M,\phi)|$ is less than $C_M$.
\end{theorem}

While Cheeger and Gromov's proof only shows the existence of such bound $C_M$, Cha found an explicit $C_M$ when we are given a triangulation of $M$.
Especially when $M$ is a zero
framed surgery on a knot, then $C_M$ can be chosen as a constant multiple of the crossing number of the knot as follows:

\begin{theorem}\cite[Theorem 1.9]{cha2014a}\label{theorem:explicit-cheeger-gromov}
Let $K$ be a knot and $c(K)$ be the crossing number of $K$. Suppose $M$ is a 3-manifold obtained by zero
framed surgery on $K$. Then
\[|\rho(M,\phi)|<69713280\cdot c(L)\]
for any homomorphism $\phi\colon\pi_1(M)\to G$ into any group $G$.
\end{theorem}

The additive property of $L^2$-signature defects under the union of 4-manifolds will be used several times:

\begin{theorem}[Novikov additivity]\label{thm:Novikov-additivity}
Let $V=V^1\sqcup V^2$ be a boundary connected sum of two 4-dimensional manifolds $V^1$ and $V^2$.
Then for any homomorphism $\phi\colon\pi_1( V)\to G$,
\[S(V,\phi)=S(V^1,\phi)+S(V^2,\phi)\]
where the induced homomorphisms on $V^1$ and $V^2$ from $\phi$ are also denoted by $\phi$.
\end{theorem}

Recall that the Levine-Tristram signature function $\sigma_K$ of a knot $K$ is a function on $S^1\subset \mathbb{C}$ defined as
\[\sigma_K(\omega)=\operatorname{sign}((1-\omega)A+(1-\ol{\omega})A^T),\]
 where A is a Seifert matrix of $K$.
It is known that abelian $\rho$-invariants of $M(K)$ can be considered as \emph{the average} of the Levine-Tristram signature function of $K$:

\begin{lemma}\cite[Proposition~5.1]{cochran2004structure}\cite[Corollary 4.3]{friedl2005l2}\label{lem:rho-invariant-and-Levine-Tristram-invariant}
Let $K$ be a knot with the meridian $\mu$ and $\phi\colon\pi_1(M(K))\to G$ be a homomorphism whose image is contained in an abelian subgroup of $G$. Then
\[
\rho(M(K),\phi) = \begin{cases} \int_{S^1} \sigma_{K}(\omega)\,d\omega &
    \text{if $\phi([\mu])\in G$ has
      infinite order,} \\[1ex]
    \sum_{r=0}^{d-1} \sigma_{K}(e^{2\pi r\sqrt{-1}/d}) &\text{if
      $\phi([\mu])\in G$ has finite order $d$.}
  \end{cases}
  \]
\end{lemma}\label{c}

\section{Construction of knots bounding gropes}\label{section:construction}

In this section, we construct infinitely many negatively amphichiral knots bounding gropes of height $n+2$ in $D^4$, with vanishing PTFA $L^2$-signature obstructions. 
Recall that a knot $K$ is called negatively amphichiral if $K$ is isotopic to $-K$, which implies $K$ has order $1$ or $2$ in the knot concordance group.
Hence knots constructed in this section can serve as prototypes for various order $2$ knots in our main theorems. 


\subsection{Infection of a knot}

The basic ingredient of our construction is the \emph{infection} of knots, which also has been called \emph{companion}, \emph{genetic modification} or \emph{satellite construction}.
We recall its definition briefly for completeness and arrangement of notations.
More details about infection can be found in \cite[Section~4.D]{rolfsen2003knots} and \cite[Section~3]{cochran2004structure}.

Let $K\subset S^3$ be a knot and $X(K)$ the knot exterior of $K$.
Let $\eta_i\colon S^1\hookrightarrow X(K)\subset$, $r=1,\dots,r$, be curves whose union forms an unlink in $S^3$.
Let $J_i$, $i=1,2,\dots, r$, be knots.
For each $i$, remove a tubular neighborhood of $\eta_i$ and instead glue the knot exterior $X(J_i)$ along their boundaries, in such a way that the meridian of $\eta_i$ is identified with the reverse of the longitude of $J_i$ and the longitude of $\eta_i$ is identified with the meridian of $J_i$.
The resulting 3-manifold is homeomorphic to $S^3$, but the the knot type of $K$ may have been changed.
It is said that $K$ is \emph{infected} by $J_i$'s along $\eta_i$'s and the resulting knot is denoted by $K_{\eta_1,\dots,\eta_r}(J_1,\cdots,J_r)$.
We call $K$ the \emph{seed knot}, $\eta_i$'s the \emph{axes}, and $J_i$'s the \emph{infection knots}.

Infection has been utilized to construct new slice knots, knots bounding gropes of arbitrary height in $D^4$, or $(n)$-solvable knots by the virtue of the following results:

\begin{proposition}\label{proposition:infection-by-slice-is-slice}
Let $K$ and $J_1,\dots,J_r$ be slice knots.
Then for any axis $\eta_1\dots,\eta_r\subset X(K)$, $K_{\eta_1,\dots,\eta_r}(J_1,\dots,J_r)$ is also slice.
\end{proposition}

\begin{proposition}\cite[Theorem 3.8]{cochran2007knot} \label{theorem:satellite-bound-grope}
Let $K$ be a slice knot.
Suppose each $\eta_i$ bounds a grope of height $n$ in the exterior of $K$ and each $J_i$ is a connected sum of the Horn's knots (see Figure~\ref{figure:P_m}).
Then $K_{\eta_1,\dots,\eta_r}(J_1,\dots,J_r)$ bounds a grope of height $n+2$ in~$D^4$.
\end{proposition}

\begin{figure}[htb]
    \labellist
    \hair 0mm
    \pinlabel{$U$} at 215 130
    \pinlabel{$\dots$} at 88 9
    \pinlabel{$T$} at 190 29
    \pinlabel{$m$} at 91 -5
    \endlabellist
    \centering
    \includegraphics{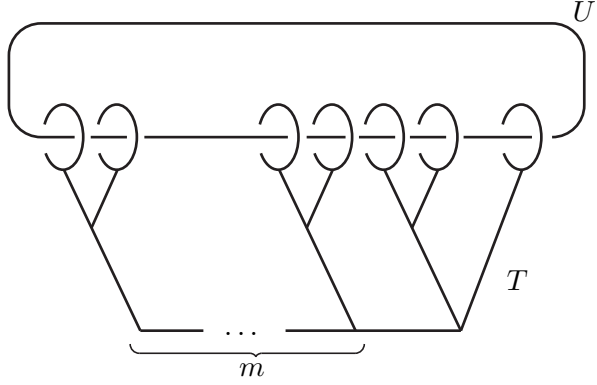}
    \caption{A clasper construction of $P_m$}
    \label{figure:P_m}
\end{figure}

Figure~\ref{figure:P_m} describes the knot $P_m$ of P.~Horn in \cite{horn2010the} by using the language of clasper surgery of K.~Habiro \cite{habiro2000claspers}.
The knot $P_m$ is obtained by performing clasper surgery on the unknot $U$ along the tree $T$.
(For a surgery description of $P_m$, see Figure~3 of \cite{horn2010the}.)

\begin{proposition}\cite[Proposition~3.1]{cochran2004structure}\label{prop:satellite-solvable}
Let $K$ be a slice knot.
Suppose the homotopy class $[\eta_i]$ of $\eta_i$ lies in $\pi_1(X(K))^{(n)}$, and $J_i$'s are $(0)$-solvable knots for all $i$.  Then $K_{\eta_1,\dots,\eta_r}(J_1,\dots,J_r)$ is $(n)$-solvable.
\end{proposition}

Here we present the proof of the last proposition briefly for later use.

\begin{proof}[Sketch of the proof]
For any $(0)$-solvable knot $J_i$, $i=1,\dots,r$, there is a $(0)$-solution $W_i$ with $\pi_1(W_i)$ isomorphic to $\Z$. (For its proof, see, e.g., Proposition 4.4 of \cite{cha2010amenable}.)
Let $W_0$ be a slice disk exterior of $K$.
Glue $W_i$ to $W_0$ along the solid torus in $\d W_i=M(J_i)$ attached during the zero
framed surgery and the tubular neighborhood of $\eta_i$ in $\d W_0=M(K)$.
Then the resulting 4-manifold is an $(n)$-solution for $K_{\eta_1,\dots,\eta_r}(J_1,\dots,J_r)$.
\end{proof}

\subsection{An iterated infection}\label{subsection:iterated-infection}

Let $\mathbb{E}_m$ be the negatively amphichiral knot in Figure \ref{figure:E_m}.
The box with label $\pm m$ denotes $m$ full positive (negative, resp.) twists between bands but individual bands left untwisted.
Let $K$ be the connected sum of two copies of $\mathbb{E}_m$ with two axes $\a, \b\subset X(K)$ as in the Figure~\ref{figure:K}.
Since $\mathbb{E}_m$ is negatively amphichiral, $K$ is a ribbon knot.
Note that $\a\sqcup\b$ forms an unlink in $S^3$ and the linking numbers $\lk(\a,K)$ and $\lk(\b,K)$ are zero.
For each $k=1,\dots,n-1$, let $K^k$ be a ribbon knot and $\eta_k\subset X(K^k)$ be an axis with $\lk(K^k,\eta_k)=0$.
Let $J^0$ be any knot bounding a grope of height 2 in~$D^4$.

Define inductively
\[J^{k}= K^{k}_{\eta_k}(J^{k-1})\]
for $1\leq k \leq n-1$, and define
\[J^n= K_{\a,\b}(J^{n-1},\ol{J^{n-1}}),\]
where $\ol{J^{n-1}}$ is the mirror image of $J^{n-1}$.
It is shown that $J^n$ is negatively amphichiral in Lemma~2.1 of \cite{cochran20112-torsion}.
We show that $J^n$ bounds a grope of height $n+2$ in $D^4$.

\begin{figure}[htb]
    \labellist
    \hair 0mm
    \pinlabel{$m$} at 13 47
    \pinlabel{$-m$} at 195 47
    \endlabellist
    \centering
    \includegraphics{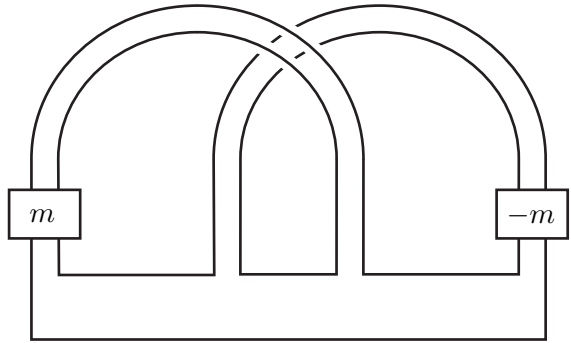}
    \caption{$\mathbb{E}_m$}
    \label{figure:E_m}
\end{figure}

\begin{figure}[htb]
    \labellist
    \hair 0mm
    \pinlabel{$m$} at 13 53
    \pinlabel{$-m$} at 195 53
    \pinlabel{$m$} at 227 53
    \pinlabel{$-m$} at 409 53
    \pinlabel{$\a$} at 32 36
    \pinlabel{$\b$} at 176 36
    \endlabellist
    \centering
    \includegraphics{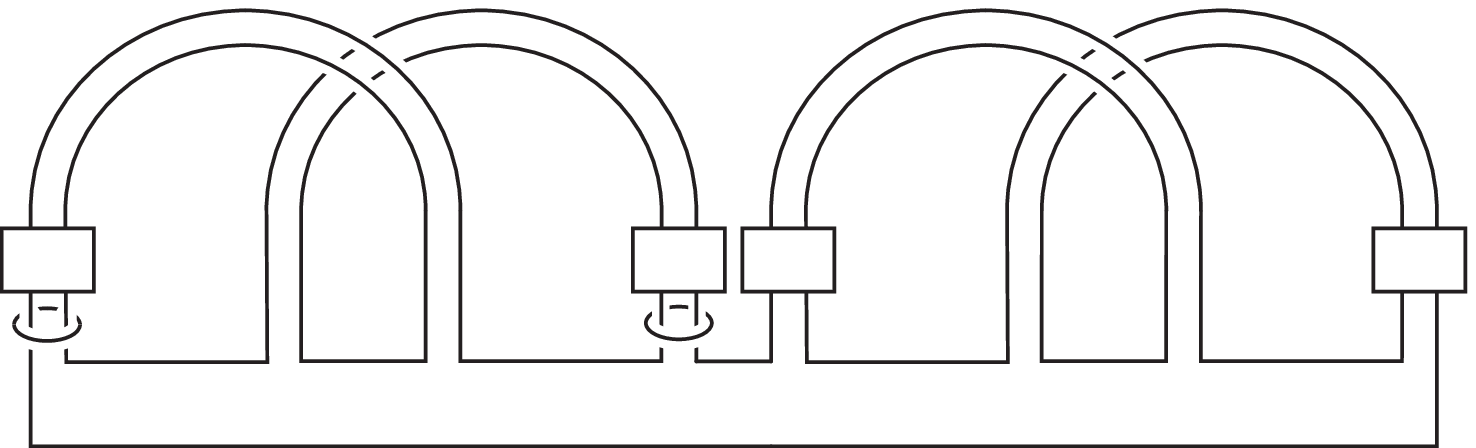}
    \caption{$K$}
    \label{figure:K}
\end{figure}

Let $U^0$ be an unknot, and define inductively
\[U^{k}= K^{k}_{\eta_k}(U^{k-1})\]
for $1\leq k \leq n-1$, and
\[U^n= K_{\a,\b}(U^{n-1},\ol{U^{n-1}}).\]

Suggested by Cha~\cite[Section~4.2]{cha2010amenable}, $X(U^n)$ can be decomposed as
\[X(U^0)\sqcup X(\ol{U^0})\sqcup X(K^1\sqcup \eta_1)\sqcup X(\ol{K^1\sqcup \eta_1})\sqcup\dots X(K^{n-1}\sqcup \eta_{n-1})\sqcup X(\ol{K^{n-1}\sqcup \eta_{n-1}})\sqcup X(K\sqcup\a\sqcup\b),\]
where $\ol{K^i\sqcup \eta_i}$ is the mirror image of $K^i\sqcup \eta_i$, considered as a 2-component link.
Note that $U^n$ is a slice knot by Proposition~\ref{proposition:infection-by-slice-is-slice}.
It can be shown that $\eta_1$ and $\ol{\eta_1}$, considered as curves in $X(U^n)$, represent elements in $\pi_1(X(U^n))^{(n)}$ by following the same argument as in the proof of Lemma~4.9 of \cite{cha2010amenable}.

On the other hand, it is easy to see that $J^n$ is isotopic to $U^n_{\eta_1,\ol{\eta_1}}(J^0,\ol{J^0})$.
By Proposition~\ref{theorem:satellite-bound-grope}, $J^n$ bounds a grope of height $n+2$ in $D^4$.

\subsection{Infection knots with a height 2 grope and vanishing signature integral}\label{subsection:K0}

In this subsection we construct a special set of infection knots $J^0_i$:
\begin{proposition}\label{prop:knot-of-cha}
For an arbitrary constant $C>0$, there is a family of knots $J^0_1, J^0_2, \dots$ and an increasing sequence of odd primes $d_1,d_2,\dots$ satisfying the following properties:
\begin{enumerate}
\item $J^0_i$ bounds a grope of height 2 in~$D^4$,
\item $\sigma_{J^0_i}(\omega_i)=\sigma_{J^0_i}(\omega_i^{-1})>C$ and $\sigma_{J^0_i}(\omega_i^r)=0$ for $r\not\equiv\pm1$ (mod $d_i$),
\item $\sigma_{J^0_j}(\omega_i^r)=0$ for any $r$, whenever $i>j$, and
\item \label{subprop}$\Arf(J^0_i)=0$ and $\int_{S^1}\sigma_{J^0_i}(\omega)\,d\omega=0$,
\end{enumerate}
where $\omega_i=e^{2\pi\sqrt{-1}/d_i}\in\mathbb{C}$ is a primitive $d_i$-th root of unity. 
\end{proposition}
Its proof resembles that of \cite[Proposition 4.12]{cha2010amenable}, which shows the existence of knots satisfying the condition (2), (3), and (4).
\begin{proof}
The Levine-Tristram signature function of Horn's knots $P_m$ in Figure~\ref{figure:P_m} is calculated by~Horn:
\begin{proposition}\cite[Proposition~3.1, Lemma~6.1]{horn2010the}
Each $P_m$ bounds a grope of height 2 in $D^4$ whose Levine-Tristram signature function $\sigma_{P_m}$ is as follows:
\[
\sigma_{P_m}(\theta)= \begin{cases} 2 & \text{if } \theta_m<\mid\theta\mid\le\pi, \\[1ex]
    0 &\text{if }0\le\mid\theta\mid<\theta_m,\\[1ex]
	\end{cases}
  \]
where $\theta_m$ is a real number in $(0,\pi)$ such that
\[\cos(\theta_m)=\frac{2\sqrt[3]{m}-1}{2\sqrt[3]{m}}.\]
\end{proposition}

Choose an increasing sequence $m_1, m_2, \dots$ of positive integers such that there is at least one prime number, say $d_i$, between $2\pi/\theta_{m_i}$ and $2\pi/\theta_{m_{i+1}}$.
Then the knot $S_i:=-P_{m_i}\#P_{m_{i+1}}$ has the bump Levine-Tristram signature function supported by neighborhoods of $\omega_i:=e^{2\pi\sqrt{-1}/d_i}$ and $\ol{\omega}_i$.
Note that for all $i$, the supports $\sigma_{S_i}^{-1}(\Z-\{0\})\subset S^1$ are disjoint each other. 

Let $S_i^\prime$ be the $(d_i, 1)$-cable of $S_i$, and let $J_i^\prime=S_i\#(-S_i^\prime)$.
By the property of the Levine-Tristram signature function of the cable of a knot (Lemma~4.13 of \cite{cha2010amenable}), we have $\sigma_{J^\prime_i}(\omega^k_j)=\sigma_{S_i}(\omega^k_j)-\sigma_{S_i}(\omega^{kd_i}_j)$.
Therefore, for $j=i$, we have $\sigma_{J^\prime_i}(\omega^k_i)=\sigma_{S_i}(\omega^k_i)-\sigma_{S_i}(1)=\sigma_{S_i}(\omega^k_i)$.
Thus, $\sigma_{J^\prime_i}(\omega^k_i)\neq0$ if and only if $k\equiv\pm1$ mod $d_i$.
Also, for $j<i$, we have $\sigma_{J^\prime_i}(\omega^k_j)=0$ since $\sigma_{S_i}(\omega^k_j)$ vanishes for any~$k$.

Also we have $\int_{S^1}\sigma_{S_i}(\omega)d\omega=\int_{S^1}\sigma_{S_i^\prime}(\omega)d\omega$.
It follows that $\int_{S^1}\sigma_{J^\prime_i}(\omega)d\omega=0$.
Now the desired knot $J^0_i$ is obtained by taking the connected sum of sufficiently large number of copies of~$J_i^\prime$.
\end{proof}
We use $J^0_i$ in place of $J^0$ in the construction of negatively amphichiral knots bounding gropes of height $n+2$ in $D^4$ in Subsection~\ref{subsection:iterated-infection} and call the resulting knot $J_i\equiv J^n_i$ for simplicity.
These knots are the prototype of knots in our main theorems (Theorems~\ref{theorem:maintheorem1}, \ref{theorem:maintheorem3} and~\ref{theorem:maintheorem2}).
The constant $C$ in Proposition~\ref{prop:knot-of-cha} will be specified explicitly in the following sections.

\subsection{Vanishing higher-order PTFA $L^2$-signature obstructions}\label{subsection:vanishing-PTFA-condition}

Here we show that each $J_i$ in Subsection~\ref{subsection:K0} are $(n)$-solvable with vanishing PTFA $L^2$-signature obstructions (Definition~\ref{def:vanishing-l2-signature-obstruction}), that is, there is an $(n)$-solution $V$ for $J_i$ such that for any group homomorphism $\phi\colon\pi_1(M(J_i))\to G$ with a PTFA group $G$ which extends to $\pi_1(V)$, $\rho(M(J_i),\phi)$ is equal to zero.

According to the proof of Proposition~\ref{prop:satellite-solvable}, $J_i$ is $(n)$-solvable with a special $(n)$-solution $V$ which is the union of a $(0)$-solution $V_1$ for $J^0_i$, $-V_1$ the orientation-reversed $V_1$, and $V_0$ a slice disk exterior of $U^n$.
For any group homomorphism $\phi\colon \pi_1(M(J_i))\to G$ with a PTFA group $G$ which extends to $\pi_1(V)$, by the Novikov additivity of $L^2$-signature defects (Theorem~\ref{thm:Novikov-additivity}), $\rho(M(J_i),\phi)=S(V,\phi)$ is equal to
\[S(V_0,\phi)+S(V_1,\phi)+S(-V_1,\phi).\]

Since $S(V_0,\phi)=\rho(M(U^n),\phi)$ and $U^n$ is a slice knot, by Theorem~1.2 of \cite{cha2010amenable}, $S(V_0,\phi)$ is equal to zero.
On the other hand, since $\pi_1(V_1)=\Z$, the image of  $\phi$ restricted to $\pi_1(M(J^0))$ lies in a cyclic subgroup of $ G $. Since $ G $ is torsion-free, this cyclic subgroup must be trivial or isomorphic to $\Z$.
By Theorem \ref{lem:rho-invariant-and-Levine-Tristram-invariant}, $S(V_1,\phi)=\rho(M(J^0_i),\phi)$ is equal to zero or the integral of the Levine-Tristram signature function of $J^0_i$ over $S^1$, which is also equal to zero by Proposition~\ref{prop:knot-of-cha}~(\ref{subprop}).
Similarly, $S(-V_1,\phi)$ is also zero.
Hence $\rho(M(J_i),\phi)$ vanishes, and this finishes the proof.

\section{Proof of Theorems \ref{theorem:maintheorem1} and \ref{theorem:maintheorem3}}\label{section:proof-of-theorem-A}

\subsection{Proof of Theorem \ref{theorem:maintheorem1} modulo infection axis analysis}\label{subsection:proof-of-theorem-A}
We construct knots $J_1,J_2,\dots$ using the iterated infection construction in Section~\ref{section:construction} with the following choice of $K$, $K^k$, and~$J^0_i$'s:

\begin{enumerate}
\item[$K$:]
Let $K$ be the connected sum of two copies of $\mathbb{E}_1$ as in Figure~\ref{figure:K}.
Take curves $\a$ and $\b$ in Figure \ref{figure:K} as axes.
\item[$K^k$:]
For any $k=1, \dots,n-1$, let $K^k$ be the knot in Figure~\ref{figure:K^k}.
Note that $K^k$ is a ribbon knot with cyclic Alexander module $\Z[t^{\pm1}]/\langle 2t-5+2/t\rangle$, which is generated by $\eta_k$ in Figure~\ref{figure:K^k}.
We take $\eta_k$ as axes.

\begin{figure}[htb]
\labellist
\hair 0mm
\pinlabel{$\eta_k$} at 102 43
\endlabellist
  \centering
\includegraphics{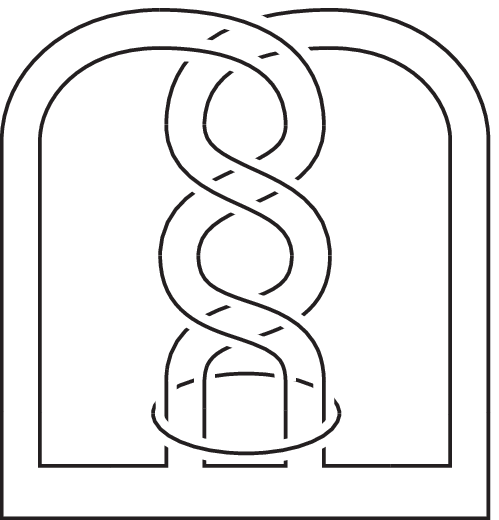}
   \caption{$K^k$}
  \label{figure:K^k}
\end{figure}

\item[$J^0_i$:] By Theorem~\ref{thm:cheeger-gromov-bound}, there is a constant which is greater than
\[\bigg|\rho(M(K),\phi_K)+\sum_{k=1}^{n-1} \rho(M(K^k),\phi_k)+\sum_{k=1}^{n-1} \rho(M(\ol{K^k}),\ol{\phi_k})\bigg|, \]
for any choice of homomorphisms $\phi_K$, $\phi_k$ and $\ol{\phi_k}$, $k=1,\dots,n-1$.
Note that $K$ has the crossing number at most 16 and $K^k$ and $\ol{K^k}$ has at most 12 for all $k$.
By Theorem~\ref{theorem:explicit-cheeger-gromov}, $69713280(16+24(n-1))$ can be an explicit bound of the above equation, regardless of the choices of $\phi_K$, $\phi_k$ and $\ol{\phi_k}$.
Choose this constant as $C$ in the construction of $J^0_i$ in Proposition~\ref{prop:knot-of-cha}.
\end{enumerate}

\begin{remark}
In general, there is no harm to choose $\mathbb{E}_m\#\mathbb{E}_m$ as $K$ and any slice knot with cyclic Alexander module as $K^k$.
We choose specific $K$ and $K^k$ to find a particular $C$ in the construction of~$J^0_i$.
\end{remark}

We prove that knots $J_i$ generate a $\Z_2^\infty$ subgroup of $\G_{n+2}/\G_{n+2.5}$.
It is enough to show that no nontrivial finite connected sum of knots $J_i$ is $(n.5)$-solvable.

Suppose not.
We may assume there is $r>0$ such that $J= J_1\#J_2\#\dots \#J_r$ has an $(n.5)$-solution $V$ for $J$, just by deleting a finite number of $J_i$'s.

We are going to construct a $4$-manifold $W_0$ and a group homomorphism $\vp$ on $\pi_1(W_0)$, and then compute the $L^2$-signature defect of $(W_0,\vp)$ in two different ways.
By showing two values are not equal, we get a contradiction and finish the proof of the claim.

In general, for a seed knot $\tilde{K}$, its axes $\eta_1,\dots,\eta_r$ and infection knots $\tilde{J_1},\dots,\tilde{J_r}$, there is a standard cobordism from $M(\tilde{K}_{\eta_1,\dots,\eta_r}(\tilde{J_1},\dots,\tilde{J_r}))$ to the disjoint union $M(\tilde{K})\sqcup M(\tilde{J_1})\sqcup\dots\sqcup M(\tilde{J_r})$. 
It is 
$$M(\tilde{K})\times [0,1] \,\,\sqcup\,\,\amalg_{i=1}^r(-M(\tilde{J_i})\times [0,1])/\sim,$$
where the tubular neighborhood of $\eta_i$ in $M(\tilde{K})\times\{0\}$ is identified with the solid torus $(M(\tilde{J_i})-X(\tilde{J_i}))\times\{0\}$ attached during the surgery process.
See the proof of Lemma~2.3 of \cite{cochran2009knot} for more detail.

Since a connected sum can be viewed as an infection, there is a cobordism $E_n$ from $M(J)$ to $M(J_1)\sqcup M(J_2)\sqcup \dots \sqcup M(J_r)$.
Let $E_{n-1}$ be the cobordism from $M(J_1=J^n_1)$ to $M(K)\sqcup M(J^{n-1}_1)\sqcup M(\ol{J^{n-1}_1})$.
Also Let $E_i$, $i=0,1,\dots,n-2$, be a cobordism from $M(J^{i+1}_1)$ to $M(K^{i+1}) \sqcup M(J^i_1)$.
Note that $-E_i$, the orientation reversed $E_i$, is a cobordism from $M(\ol{J^{i+1}_1})$ to $M(\ol{K^{i+1}})\sqcup M(\ol{J^i_1})$.
Finally let $V^j$, $j=2,\dots,r$, be the $(n)$-solution for $J_j$ constructed in the proof of Proposition~\ref{prop:satellite-solvable}.
That is, $V^j$ is the union of $V^j_0$, a slice exterior of $U^n$, $V^j_1$, a $(0)$-solution for $J^0_j$, and $-V^j_1$.

Define
$$W_n=(\amalgover{j=2}{r} -V^j)\amalgover{\amalgover{j=2}{r} M(J_j)}{}E_n\amalgover{M(J)}{}V$$
and
$$W_{n-1}=E_{n-1}\amalgover{M(J_1)}{}W_n,$$
where 4-manifolds are glued along homeomorphic boundary components.
For $i=n-2, \dots,0$, define inductively
$$W_i=E_i\amalgover{M(J^{i+1}_1)}{}W_{i+1}\amalgover{M(\ol{J^{i+1}_1})}{}-E_i$$

Figure~\ref{figure:W_k} will be helpful to understand the construction.

\begin{figure}[tb]
\labellist
\small
\pinlabel{$M(J_1)=M(J_1^n)$} at 122 91
\pinlabel{$M(J)$} at 122 29
\pinlabel{$V$} at 225 17
\pinlabel{$-V^2$} at 287 80
\pinlabel{$-V^r$} at 397 80
\pinlabel{$\dots$} at 344 75
\pinlabel{$E_n$} at 225 65
\pinlabel{$E_{n-1}$} at 225 112
\pinlabel{$E_{n-2}$} at 28 150
\pinlabel{$-E_{n-2}$} at 225 150
\pinlabel{$M(J_1^{n-1})$} at 48 132    
\pinlabel{$M(\ol{J_1^{n-1}})$} at 200 132
\pinlabel{$M(K)$} at 130 147
\pinlabel{$M(K^{n-1})$} at 85 180
\pinlabel{$M(\ol{K^{n-1}})$} at 170 180   
\pinlabel{$M(J_2)$} at 287 62    
\pinlabel{$M(J_r)$} at 397 62   
\pinlabel{$M(J^{i}_1)$} at 6 245
\pinlabel{$M(K^{i+1})$} at 46 246
\pinlabel{$M(\ol{K^{i+1}})$} at 197 246
\pinlabel{$M(\ol{J^i_1})$} at 247 246
\pinlabel{$\vdots$} at 32 185
\pinlabel{$\vdots$} at 225 185
\pinlabel{$E_i$} at 45 210
\pinlabel{$-E_i$} at 239 210    
\endlabellist
\includegraphics{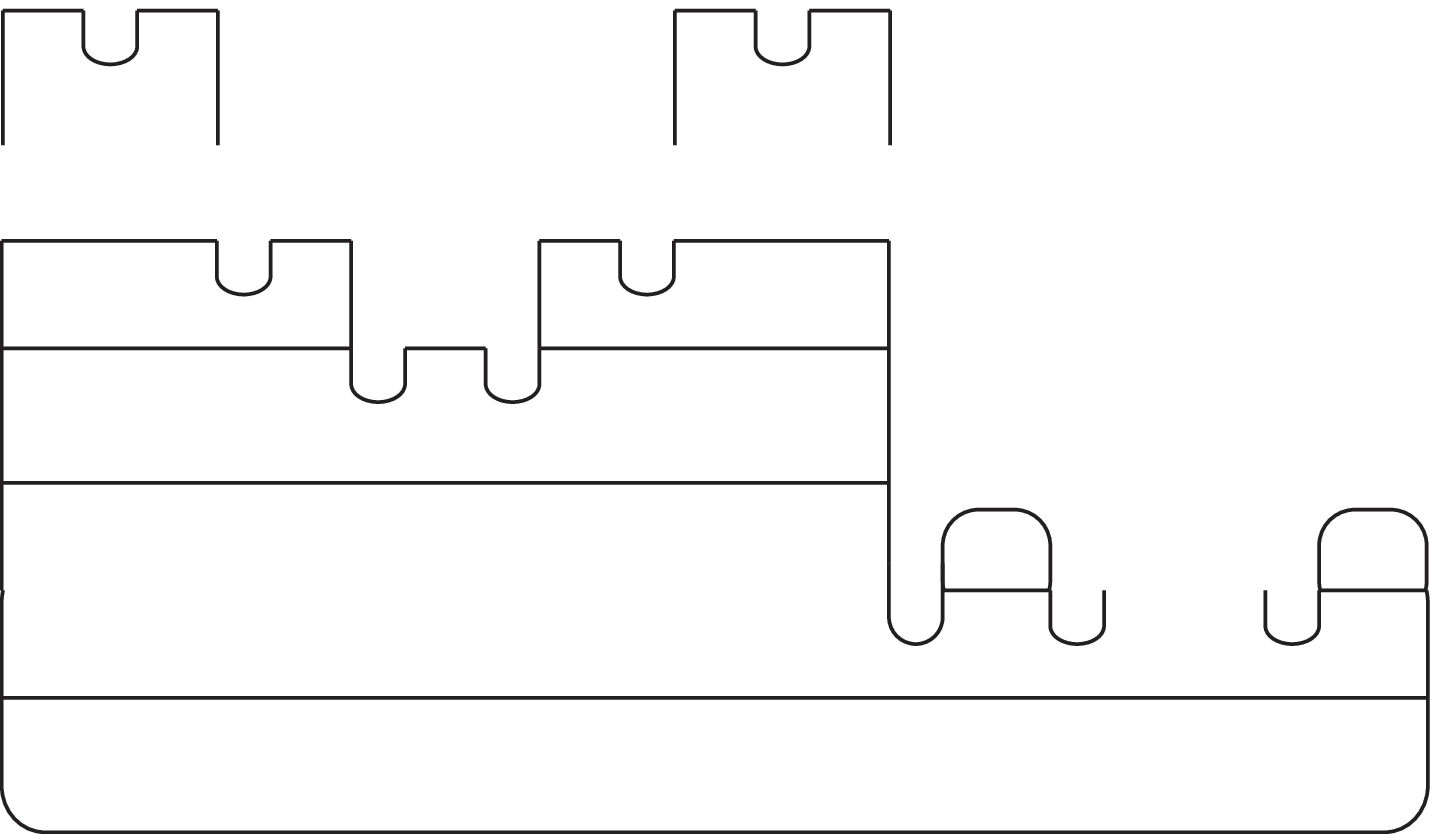}
\caption{$W_i$}
\label{figure:W_k}
\end{figure}

Now we define a homomorphism $\vp_k$ on $W_k$.
Denote $\pi_1(W_i)$ by $\pi$ for simplicity.
We construct a normal series $\{\S_k\pi\}$, $k=0,1,\dots,n+1$, of~$\pi$.

Let $\S_k\pi=\pi^{(k)}_r$ for $k=0,1$, and $2$, be the $k$-th rational derived subgroup of $\pi$.
Recall that the $k$-th rational derived subgroup $G^{(k)}_r$ of a group $G$ is defined inductively as the kernel of the map
\[G^{(k)}\to \frac{G^{(k)}}{[G^{(k)},G^{(k)}]}\otimes_\Z \Q ,\]
with the initial condition $G^{(0)}_r=G$.

By abuse of notation, we denote the element in $\pi$ represented by the curve $\b\subset M(K)$ as $\b$.
Observe that for any element $x\in\pi$, $x\b x^{-1}$ lies in the first rational derived subgroup $\pi^{(1)}_r$ of $\pi$.
Let $S$ be the subset of $\Z[\S_1\pi/\S_2\pi]\subset\Z[\pi/\S_2\pi]$ multiplicatively generated by
\[\{p(\mu^j\b\mu^{-j})|j\in\Z\},\]
where $p(t)=2t-5+2/t$ is the Alexander polynomial of $K^{n-1}$ and $\mu$ is the meridian of $K$.
We want to localize $\Z[\pi/\S_2\pi]$ with $S$.
To do that, we have to check whether $S$ is a \emph{right divisor set} since $\Z[\pi/\S_2\pi]$ may not be commutative.

A multiplicative subset $S_0$ of a (noncommutative) domain $R$ is a \emph{right divisor set of $R$} if $0$ is not in $S_0$ and for any $a\in R$ and $s\in S_0$, the intersection $aS_0\cup sR$ is nonempty.
It is well-known that the right localized ring $RS_0^{-1}$ exists if and only if $S_0$ is a right divisor set of $R$ \cite[p.427]{passman1976the}, \cite[Section~3.1]{cochran2011primary}.
The following result is useful to show that $S$ is a right divisor~set:

\begin{proposition}\cite[Proposition~4.1]{cochran2011primary}\label{proposition:right-divisor-set}
Let $R$ be a ring and $G$ be a group. Suppose $A$ is a normal subgroup of $G$ such that the group ring $R[A]$ is a domain. If $S_0$ is a right divisor set of $R[A]$ that is $G$-invariant ($g^{-1}S_0g=S_0$ for all $g\in G$), then $S_0$ is a right divisor set of $R[G]$.
\end{proposition}

Note that $S$ is a right divisor set of $\Z[\S_1\pi/\S_2\pi]$ since $\Z[\S_1\pi/\S_2\pi]$ is a commutative ring.
Also $S$ is $\pi/\S_2\pi$-invariant by definition.
By Proposition~\ref{proposition:right-divisor-set}, $S$ is a right divisor set of  $\Z[\pi/\S_2\pi]$ and a left $\Z[\pi/\S_2\pi]$-module $\Z[\pi/\S_2\pi]S^{-1}$ is well-defined.

Recall $\S_2\pi/[\S_2\pi,\S_2\pi]$ has the right $\Z[\pi/\S_2\pi]$-module structure induced by conjugation action of $\pi/\S_2\pi$ on $\S_2\pi/[\S_2\pi,\S_2\pi]$
($x\ast g=g^{-1}xg$ for any $g\in\pi/\S_2\pi$ and $x\in\S_2\pi/[\S_2\pi,\S_2\pi]$).
Define \[\S_3\pi=\ker\{\S_2\pi\to\frac{\S_2\pi}{[\S_2\pi,\S_2\pi]}\otimes_{\Z[\pi/\S_2\pi]}\Z[\pi/\S_2\pi]S^{-1}\}.\]

For $3<k<n+1$, define $S_k= (S_3\pi)^{(k-3)}_r$ and for $k=n+1$, define 
\[\S_{n+1}\pi=\ker\{\S_n\pi\to \frac{\S_n\pi}{[\S_n\pi,\S_n\pi]}\otimes_\Z \Z_d\},  \]
where $d=d_1$ is the prime number used in the construction of $J^0_1$ (see Proposition~\ref{prop:knot-of-cha}).

Let $G_k =\pi/\S_{n-k+1}\pi$ and $\vp_k\colon \pi=\pi_1(W_i)\to  G_k$ be the quotient map.
We denote $\vp_0$ by $\vp$ and $G_0$ by $G$ for simplicity.
We are going to compute $L^2$-signature defect $S(W_0,\vp)$ in two ways, which give different values.

\vskip 5mm
\emph{First method.} By the Novikov additivity theorem (Theorem~\ref{thm:Novikov-additivity}), $S(W_0,\vp)$ is equal to
$$S(V,\vp)+S(E_n,\vp)+S(E_{n-1},\vp)+\sum_{k=0}^{n-2}(S(E_k,\vp)+S(-E_k,\vp))-\sum_{i=2}^r S(V^i,\vp),$$
where all the homomorphisms induced from $\vp$ are denoted by $\vp$ for simplicity.
We compute each summand as follows:

(1) $G$ has a subnormal series
\[0=G^{n+1}\lhd G^n \lhd \dots \lhd G^1\lhd G^0=G,\]
where $G^i\equiv \S_i\,\pi/\S_{n+1}\,\pi<\pi/\S_{n+1}$. Since $G^i/G^{i+1}$ is abelian and has no torsion coprime to $d$, $G$ is amenable and in $D(\Z_d)$ by Lemma~\ref{lemma:amenable_group}. 
By applying Theorem \ref{thm:homology-cobordism-invariance-on-amenable} on $(V,\vp)$ we get $S(V,\vp)=0$.

(2) Lemma 2.4 of \cite{cochran2009knot} implies that $S(E_k,\vp)$  and $S(-E_k,\vp)$, $k=0,\dots,n$ are all zero.

(3) Again by the Novikov additivity theorem, $S(V^j,\vp)$ is equal to
\[S(V^j_0,\vp)+S(V^j_1,\vp)+S(-V^j_1,\vp),\]
where $V^i_0$ is a slice exterior of $U^n$, and $V^i_1$ is a $(0)$-solution for $J^0_i$ with $\pi_1(V^i_1)=\Z$, respectively.
Similar calculation to Subsection~\ref{subsection:vanishing-PTFA-condition} gives that $S(V^i_0,\vp)=0$.
Since the meridian of $J^0_i$ lies in the $n$-th derived subgroup of $\pi$, the image of $\vp$ lies in $G^n$, an abelian subgroup of $G$ which consists of elements of order $1$ and $d$.
By Lemma~\ref{lem:rho-invariant-and-Levine-Tristram-invariant}, $S(V^i_1,\vp)$ is equal to 0 or
\[\sum_{r=0}^{d-1} \sigma_{J^0_i}\left(e^{2\pi r\sqrt{-1}/d}\right),\]
which is again 0 by the choice of $J^0_i$ (Proposition~\ref{prop:knot-of-cha}, condition (2)).
Similarly, $S(-V^i_1,\vp)$ is also equal to~0.

In sum, we conclude that $S(W_0,\phi)$ is equal to 0.

\vskip 5mm
\emph{Second method.} By the von Neumann index theorem, the $L^2$-signature defect of a 4-manifold is equal to the $\rho$-invariant of the boundary with the inclusion-induced homomorphism.	
Hence $S(W_0,\vp)$ is equal to
\begin{equation}\label{equation:L2-signature-defect-of-W0}
\rho(M(K),\vp)+\sum_{i=1}^{n-1}(\rho(M(K^i),\vp)+\rho(M(\ol{K^i}),\vp))+\rho(M(J^0_1),\vp)+\rho(M(\ol{J^0_1}),\vp).
\end{equation}
For the estimation of this value, the following lemma is crucial:

\begin{lemma}\label{Lemma:nontriviality}
For $k=0,1,\dots,n-1$, the homomorphism
\[\vp_k\colon \pi_1(W_k)\to \frac{\pi_1(W_k)}{\S_{n-k+1}\pi_1(W_k)}\]
sends the meridian $\mu_k$ of $J^k_1$ and that of $\ol{\mu}_k$ of $\ol{J^k_1}$ into the subgroup $S_{n-k}\pi_1(W_k)/\S_{n-k+1}\pi_1(W_k)$.
Also, $\vp_k(\mu_k)$ is nontrivial for any $k=0,\dots,n-1$, while $\vp_k(\ol{\mu}_k)$ is nontrivial for $k=n-1$ and trivial for other~$k$.
\end{lemma}

Assuming Lemma \ref{Lemma:nontriviality} is true, we finish the proof of Theorem \ref{theorem:maintheorem1}.
For simplicity, denote $\pi_1(W_0)$ as $\pi$.
Note that, by the definition of $\S_k$, there is a natural inclusion
\[\S_{n}\pi/\S_{n+1}\pi\hookrightarrow\frac{\S_{n}\pi}{[\S_{n}\pi,\S_{n}\pi]}\otimes_\Z\Z_d.\]
Hence $\S_{n}\,\pi/\S_{n+1}\pi$, which contains the image of $\mu_0$ and $\ol{\mu}_0$, is an abelian subgroup of $\pi/\S_{n+1}\pi$ and every nontrivial element in $\S_n\pi/\S_{n+1}\pi$ has order $d$.
By Lemma \ref{Lemma:nontriviality} $\vp(\mu_0)$ has order~$d$.

Now we can apply Lemma~\ref{lem:rho-invariant-and-Levine-Tristram-invariant} for the calculation of $\rho(M(J^0_1),\vp)$ and $\rho(M(\ol{J^0_1}),\vp)$.
By the choice of $J^0_1$ and $C$,
\[\rho(M(J^0_1),\vp)=\sum_{r=0}^{d-1} \sigma_{J^0_1}(e^{2\pi r\sqrt{-1}/d})>C>|\rho(M(K),\vp)+\sum_{i=1}^{n-1}(\rho(M(K^i),\vp)+\rho(M(\ol{K^i}),\vp))|\]

On the other hand, $\vp(\ol{\mu}_0)$ is zero, so $\rho(M(\ol{J^0_1}),\vp)$ vanishes.
Hence by Equation~(\ref{equation:L2-signature-defect-of-W0}), $S(W_0,\vp)$ is strictly greater than~$0$.

\vskip 5mm
Based on the above two contradictory calculations of $S(W_0,\vp)$, we conclude that any nontrivial finite connected sum of $J_i$'s cannot vanish in $\F_n/\F_{n.5}$. 
Hence to finish the proof of Theorem~\ref{theorem:maintheorem1} it remains to prove Lemma~\ref{Lemma:nontriviality} to whom the next subsection is devoted.

\subsection{Nontriviality of the infection axes}

In this subsection we prove Lemma~\ref{Lemma:nontriviality}.
Results on \emph{higher-order Blanchfield linking forms} are used at the crucial step so we first introduce them and proceed the proof.
More detailed argument can be found at Section~5 of \cite{cha2010amenable}.
\vskip 5mm

\subsubsection*{Higher-order Blanchfield forms}

Suppose that there are a commutative ring $R$ with unity and a closed 3-manifold $M$ endowed with a group homomorphism $\phi\colon\pi_1(M)\to G$ satisfying:
\begin{enumerate}
\item[(BL1)] the group ring $R[G]$ is an Ore domain, that is, there is the quotient skew-field $\mathcal{K}\equiv R[G](R[G]-\{0\})^{-1}$ of~$R[G]$.
\item[(BL2)] $H_1(M,\mathcal{K})$ vanishes.
\end{enumerate}
Then for any localization ring $\R$ of the group ring $R[G]$ such that $R[G]\subset \R\subset\mathcal{K}$, there is a bilinear form over~$\R$,
\[\Bl \colon H_1(M,\R)\times H_1(M,\R)\to \mathcal{K}/\R\]
which is called the \emph{higher-order Blanchfield linking form on $H_1(M,\R)$}.

Note that if $G$ is PTFA and $R$ is $\Z$ or $\Z_d$ for some prime $d$, and $M$ is the zero framed surgery $M(K)$ on a knot $K$ endowed with nontrivial $\phi$,  then the condition (BL1) an (BL2) are satisfied (see Section~5.1 of \cite{cha2010amenable}).

The following two theorems on Blanchfield linking forms are main tools on the proof of Lemma~\ref{Lemma:nontriviality}.
The proof for $R=\Z_d$ case is exactly the same with $R=\Z$ case which is proved in \cite{cochran2011primary,cochran2009knot}, so we omit the proof.

\begin{theorem}\cite[Lemma~7.16 for $R=\Z$]{cochran2011primary}\label{lemma:nonsingularity}
Let $K$ be a knot and $\phi\colon\pi_1(M(K))\to G$ be a group homomorphism with $G$ a PTFA group which factors nontrivially through $\Z$. 
Suppose $\R=(R[G])S^{-1}$ for $R=\Z$ or $\Z_d$ for some prime $d$ is an Ore localization where $S$ is closed under the natural involution on  $R[G]$. 
Then the Blanchfield linking form $\Bl$ on $H_1(M,\R)$ is non-singular.
\end{theorem}

The following is a generalization of \cite[Theorem~6.3]{cochran2009knot}.
For the definition of \emph{$R$-coefficient $(k)$-bordisms} which appear in this theorem, see  \cite[Section~5]{cochran2009knot} or \cite[Definition~5.5]{cha2010amenable}. 
We remark that all $W_i$ are $\Z$ or $\Q$-coefficient $(k)$-bordisms for any $0\leq i \leq n$, $1\leq k\leq n$, and $W_0$ is a $\Z_d$-coefficient $(n)$-bordism for any prime $d$ (see \cite[Lemma~5.7]{cha2010amenable}).

\begin{theorem}\cite[Theorem 4.4]{cha2010amenable}\label{theorem:cochran-harvey-leidy-theorem}
Let $R$ be $\Z$ or $\Z_d$ and $W$ be an $R$-coefficient $(k)$-bordism.
Let $\phi\colon\pi_1(W)\to G$ be a nontrivial group homomorphism where $G$ is a PTFA group with $G^{(k)}=1$.
Suppose, for each connected component $M_i$ of $\d W$ for which $\phi$ restricted to $\pi_1(M_i)$ is nontrivial, that $\b_1(M_i)=1$.
For any localization $\R$ of $R[G]$, let $P$ be the kernel of the inclusion-induced~map
\[H_1(\d W,\R)\to H_1(W,\R).\]
Then $P\subset P^\perp$ with respect to the Blanchfield form on $H_1(\d W,\R)$.
\end{theorem}

Now we prove Lemma~\ref{Lemma:nontriviality}.

\vskip 5mm

\subsubsection*{Proof of Lemma~\ref{Lemma:nontriviality}}

Note that the boundary of $W_k$ is the disjoint union of the zero framed surgery manifolds of $K$, $K^{n-1},\dots,K^{k+1}$, $\ol{K^{n-1}},\dots,\ol{K^{k+1}}$, $J^k_1$, and $\ol{J^k_1}$.
Hence curves in the exterior of these knots, such as meridians or axes used in the infection construction, can be considered as curves in $W_k$ as well.

By abuse of notation, isotoped curves will be denoted by the same symbols.
For example, in $W_k$, $\mu_k\subset M(J^k_1)$ can be isotoped into a meridian of $K^k$, which will be denoted as $\mu_k$.
Also, the axis curve $\eta_k\subset M(K^k)$ can be isotoped into $M(J^k_1)$ and this isotoped curve will be denoted as~$\eta_k$.

We prove that the meridian $\mu_k$ of $J^k_1$ represents an element in $\pi_1(W_k)^{(n-k)}$ by using the reverse induction on $k$.
For $k=n$, the statement is obvious.
Suppose $\mu_{k+1}$ lies in the $(n-k-1)$-th derived subgroup of $\pi_1(W_{k+1})$.
Note that two curves $\mu_k$ and $\eta_{k+1}$ represent the same element in $\pi_1(W_k)$ since they are identified during the construction of $E_k$.
The isotoped curve $\eta_{k+1}\subset M(J^{k+1}_1)$ represents an element in the commutator subgroup of $\pi_1(M(J^{k+1}_1))$.
Since $\pi_1(M(J^{k+1}_1))$ is normally generated by  $\mu_{k+1}$, the inductive hypothesis  implies that $\mu_k=\eta_{k+1}$ lies in the $\pi_1(W_k)^{(n-k)}$.
Similarly, $\ol{\mu_k}$ lies in $\pi_1(W_k)^{(n-k)}$.

Since $\pi_1(W_k)^{(n-k)}\le\S_{n-k}\pi_1(W_k)$, $\vp_k(\mu_k)$ and $\vp_k(\ol{\mu_k})$ lie in $S_{n-k}\pi_1(W_k)/\S_{n-k+1}\pi_1(W_k)$, and we finish the proof of the first statement of Lemma~\ref{Lemma:nontriviality}.

Now we prove the second statement by using the reverse induction on $k$.
Observe that a Mayer-Vietoris sequence argument shows that the inclusion $M(J^n_1)\hookrightarrow W_n$ induces an isomorphism
\[\Z=H_1(M(J^n_1))\cong H_1(W_n).\]
Hence $H_1(W_n)$ is isomorphic to $\pi_1(W_n)/\S_1\pi_1(W_n)=\pi_1(W_n)/\pi_1(W_n)^{(1)}$.
This implies $\mu_n$, representing a generator of $H_1(W_n)$, does not vanish in $\pi_1(W_n)/\S_1\pi_1(W_n)$.

Next, we show that $\vp_{n-1}(\mu_{n-1})$ and $\vp_{n-1}(\ol{\mu}_{n-1})$ are nontrivial.
Recall that $K$ is the connected sum of two copies of $\mathbb{E}_1$.
Then the Alexander module $H_1(M(K),\Q[t^{\pm1}])$ of $K$ splits into two copies of the Alexander module $H_1(M(\mathbb{E}_1),\Q[t^{\pm1}])$ of $\mathbb{E}_1$.
Consider the following composition of maps:
\[
\begin{diagram}
\node{\Phi\colon H_1(M(\mathbb{E}_1),\Q[t^{\pm1}])} \arrow{e}
\node{H_1(M(K),\Q[t^{\pm1}])} \arrow{e,t}{\iota_*}
\node{H_1(W_{n-1},\Q[t^{\pm1}]),}
\end{diagram}
\]
where the homologies are twisted by abelianization maps.
Here, the first map is induced from the left copy of $\mathbb{E}_1$ of $K=\mathbb{E}_1\#\mathbb{E}_1$ (see Figure~\ref{figure:K}) and the second map is induced from the inclusion $\iota\colon M(K)\hookrightarrow W_{n-1}$.

\begin{lemma}\label{lemma:injective}
The map $\Phi$ is injective.
\end{lemma}

\begin{proof}
Note that $W_{n-1}$ is a $\Z$-coefficient $1$-bordism and the abelianization map $\pi_1(W_{n-1})\to \Z$ is nontrivial.
Theorem~\ref{theorem:cochran-harvey-leidy-theorem} implies that the Blanchfield form on $H_1(M(K),\Q[t^{\pm1}])$ vanishes over the kernel of $\iota_*$.
Also, the Blanchfield form on the Alexander module of $K$ splits as sum of two copies of the Blanchfield form on the Alexander module of $\mathbb{E}_1$.
Hence the Blanchfield form on $H_1(M(\mathbb{E}_1),\Q[t^{\pm1}])$ vanishes over the kernel of~$\Phi$.

Recall that the Alexander module of $\mathbb{E}_1$ is
\[H_1(M(\mathbb{E}_1),\Q[t^{\pm1}])\cong\frac{\Q[t^{\pm1}]}{\langle t-3+1/t\rangle},\]
which is a simple $\Q[t^{\pm1}]$-module (e.g., see \cite[Example~4.10]{cochran2011primary}). 
Hence the kernel $P$ of $\Phi$ is either zero or the entire $H_1(M(\mathbb{E}_1),\Q[t^{\pm1}])$.

Note that $\Q[t^{\pm1}]=\Z[t^{\pm1}]{(\Z-\{0\})^{-1}}$ and the natural involution of $\Z[t^{\pm1}]$ acts trivially on $\Z-\{0\}$.
Hence, by Theorem~\ref{lemma:nonsingularity}, $H_1(M(\mathbb{E}_1),\Q[t^{\pm1}])$ is nonsingular.
Since that $P$ is $H_1(M(\mathbb{E}_1),\Q[t^{\pm1}])$ implies the Blanchfield form on $H_1(M(\mathbb{E}_1),\Q[t^{\pm1}])$ vanishes, $P$ must be trivial.
That is, $\Phi$ is injective.
\end{proof}

Recall that $\mu_{n-1}$ and $\ol{\mu}_{n-1}$ are identified with the curve $\a$ and $\b\subset M(K)$ in $W_{n-1}$, respectively.
It can be shown that $\a$ and $\b$ are nontrivial in $H_1(M(\mathbb{E}_1),\Q[t^{\pm1}])$.
By Lemma~\ref{lemma:injective}, $\mu_{n-1}$ and $ \ol{\mu_{n-1}}$ represent nontrivial elements in $H_1(W_{n-1},\Q[t^{\pm1}])$.

Since $\S_2\pi_1(W_{n-1})$ is defined as the kernel of the map
\[\S_{1}\pi_1(W_{n-1}) \to \frac{\S_{1}\pi_1(W_{n-1})}{[\S_{1}\pi_1(W_{n-1}),\S_{1}\pi_1(W_{n-1})]}\otimes_{\Z} \Q\cong H_1(W_{n-1},\Q[t^{\pm1}]),\]
there is a natural injective map
\[\frac{\S_{1}\pi_1(W_{n-1})}{\S_{2}\pi_1(W_{n-1})}\hookrightarrow H_1(W_{n-1},\Q[t^{\pm1}]).\]

Hence $\mu_{n-1}$ and $\ol{\mu}_{n-1}$ represent nontrivial elements in $\S_1\pi_1(W_{n-1})/\S_2\pi_1(W_{n-1})$.

For $k\leq n-2$, suppose $\vp_{k+1}(\mu_{k+1})$ is nontrivial.
Denote $\S_{n-k}\pi_1(W_{k+1})$ by $\Lambda$ temporarily.
Let $\mathcal{R}$ be $\Z[\pi_1(W_{n-1})/\Lambda]S^{-1}$ if $k=n-2$, $\Q[\pi_1(W_{k+1})/\Lambda]$ if $0<k<n-2$, and $\Z_d[\pi_1(W_{1})/\Lambda]$ if $k=0$.
Then $\S_{n-k+1}\pi_1(W_{k+1})$ is the kernel~of 
\[\Lambda=\S_{n-k}\pi_1(W_{k+1}) \to \frac{\Lambda}{[\Lambda,\Lambda]}\otimes_{\Z[\pi_1(W_{k+1})/\Lambda]} \R\hookrightarrow H_1(W_{k+1},\mathcal{R}).\]
Hence there is an injection
\begin{equation}\label{eq3}
\frac{\S_{n-k}\pi_1(W_{k+1})}{\S_{n-k+1}\pi_1(W_{k+1})}\hookrightarrow H_1(W_{k+1},\mathcal{R}).
\end{equation}

We need the following lemma:
\begin{lemma}\label{lemma:isomorphism}
The inclusion $W_{k+1}\hookrightarrow W_k$ induces an isomorphism
\[\frac{\pi_1(W_{k+1})}{\pi_1(W_{k+1})^{(n-k+1)}}\cong \frac{\pi_1(W_k)}{\pi_1(W_k)^{(n-k+1)}}.\]
As a consequence,
\[\frac{\S_{n-k}\pi_1(W_{k+1})}{\S_{n-k+1}\pi_1(W_{k+1})}\cong \frac{\S_{n-k}\pi_1(W_k)}{\S_{n-k+1}\pi_1(W_k)}.\]
\end{lemma}

\begin{proof}
The proof is essentially the same with the proof of (7.18) in \cite{cochran2011primary}.
\end{proof}
Hence we have an inclusion from $\S_{n-k}\pi_1(W_k)/\S_{n-k+1}\pi_1(W_k)$ to $H_1(W_{k+1},\mathcal{R})$.

Suppose $\mu_k$ represents a trivial element in $\S_{n-k}\pi_1(W_k)/\S_{n-k+1}\pi_1(W_k)$.
Then $\mu_k=\eta_{k+1}$ vanishes in $H_1(W_{k+1},\mathcal{R})$.
By the inductive hypothesis, $\vp_{k+1}$ is nontrivial.
Hence, by Theorem~\ref{theorem:cochran-harvey-leidy-theorem}, $\Bl([\eta_{k+1}],[\eta_{k+1}])$, for $\Bl$ the Blanchfield form on $H_1(M(J_1^{k+1}),\mathcal{R})$, vanishes.

One can show that $H_1(M(J^{k+1}_1),\mathcal{R})$ is isomorphic to $H_1(M(K^{k+1}),\Z[t^{\pm1}])\otimes_{\Z[t^{\pm1}]}\mathcal{R}$, considering $\R$ as a $\Z[t^{\pm1}]$-module by the map $t\in\langle t\rangle \mapsto [\mu_{k+1}]\in\R$ (see \cite[Theorem~4.7]{leidy2006higher} for example).
Since $\eta_{k+1}$ generates $H_1(M(K^{k+1}),\Z[t^{\pm1}])$ as a $\Z[t^{\pm1}]$-module, $\eta_{k+1}$ generates $H_1(M(J^{k+1}_1),\mathcal{R})$ as an $\mathcal{R}$-module as well.
Theorem~\ref{lemma:nonsingularity} and the fact that the Blanchfield form on $H_1(M(J^{k+1}_1),\mathcal{R})$ vanishes assures that $H_1(M(J^{k+1}_1),\mathcal{R})$ is a zero module.

Note that $H_1(M(K^{k+1}),\mathcal{R})$ is isomorphic to
\[\frac{\Z[t^{\pm1}]}{\langle p(t)\rangle}\otimes_{\Z[t^{\pm1}]}\mathcal{R},\]
where $p(t)=2t-5+2/t$ is the Alexander polynomial of $K^{k+1}$.
Possibly except when $k=n-2$, this module is not a zero module.
We show that  $H_1(M(K^{n-1}),\mathcal{R})$ is not a zero module as well.

Suppose not.
We temporarily denote $\pi_1(W_{n-1})$ as $\Lambda$ for simplicity.
Then
\[\frac{\Z[t^{\pm1}]}{\langle p(t)\rangle}\otimes_{\Z[t^{\pm1}]}\mathcal{R}\cong \frac{\Z[\Lambda/\S_2\Lambda]}{\langle p(\a)\rangle}S^{-1}=0. \]
This implies that the unity $\ol{1\cdot e}$ represents 0 in this module where $e$ is the identity element of the group $\Lambda/\S_2\Lambda$.
Then there is $s\in S$ such that
$$1\cdot s=p(\a)\cdot f\in\Z[\Lambda/\S_2\Lambda],$$
where $f\in \Z[\Lambda/\S_2\Lambda]$.
Since $p(\a)$ and $s$ are in $\Z[\S_1\Lambda/\S_2\Lambda]$, $f$ also lies in $\Z[\S_1\Lambda/\S_2\Lambda]$.
Since $\S_1\Lambda/\S_2\Lambda$ is a torsion-free abelian group containing $\a$ and $\b$, there is a finitely generated free abelian subgroup $H$ of $\S_1\Lambda/\S_2\Lambda$ which contains $\a$ and $\b$.
Choose a basis $\{x, x_2,\dots,x_r\}$ of $H$ in which $\a=x^m$ for some $m>0$.
Then there are $m_i,n_{i,j}\in\Z$ such that $\mu^i\b\mu^{-i}=x^{m_i}x^{n_{i,2}}_2\dots x^{n_{i,r}}_r$ for all $i\in\Z$.

We can regard $\Z[H]$ as a Laurent polynomial ring in the variables $\{x,x_2,\dots,x_r\}$.
Since $p(t)$ is neither zero nor a unit, there exists a non-zero complex root $\tau$ of $p(x^m)$.
Let $\tilde{p}(x)$ be the irreducible factor of $p(x^m)$ of which $\tau$ is a root.
Since $s$ is a product of elements of the form $p(\mu^i\b\mu^{-i})$ and $p(\a)$ divides $s$ in $\Z[H]$ by the assumption, $\tilde{p}(x)$ must divide $p(x^{m_i}x^{n_{i,2}}_2\dots x^{n_{i,r}}_r)$ in $\mathbb{C}[H]$ for some $i$.
Then $p(\tau^{m_i}x^{n_{i,2}}_2\dots x^{n_{i,r}}_r)$ must vanish for every complex value of $x_2,\dots,x_r$.
This implies that $n_{i,j}$ is 0 for all $j$. Hence, we get $\mu^i\b\mu^{-i}=x^{m_i}$.

Note that $m_i\neq 0$ since $\b$ is nontrivial. Thus,
\[\mu^i\b^m\mu^{-i}=\a^{m_i}\]
for some $i,m$, and $m_i$.
This equation holds in $\S_1\Lambda/\S_2\Lambda$ too since $\Z[\S_1\Lambda/\S_2\Lambda]$ is a free $\Z$-module.
Since there is an injection 
\[\frac{\S_1\Lambda}{\S_2\Lambda}\hookrightarrow \frac{\S_1\Lambda}{[\S_1\Lambda,\S_1\Lambda]}\otimes\Q=H_1(W_{n-1},\Q[t^{\pm1}]), \]
$\mu^i\b^m\mu^{-i}=\a^{m_i}$ also holds in $H_1(W_{n-1},\Q[t^{\pm1}])$.
Note that $\mu$, $\a$, and $\b$ are in the image of the injective map  $\Phi$
$$H_1(M(\mathbb{E}_1),\Q[t^{\pm1}])\hookrightarrow H_1(W_{n-1},\Q[t^{\pm1}])$$ 
and hence the equation $\mu^i\b^m\mu^{-i}=\a^{m_i}$ remains valid in $H_1(M(\mathbb{E}_1),\Q[t^{\pm1}])$.

On the other hand, according to the Lemma~7.8 of \cite{cochran20112-torsion}, the intersection of subgroups $\langle \mu^i \b\mu^{-i}\rangle$ and $\langle \a\rangle$ of $H_1(M(\mathbb{E}_1),\Q[t^{\pm1}])$ is trivial for every $i$.
This contradiction shows that $H_1(M(J_1^{n-1}),\mathcal{R})$ is not trivial, and then completes the proof of $\mu_k$ part of Lemma~\ref{Lemma:nontriviality}.

Next we prove that $\vp_k(\ol{\mu}_k)$ is trivial for $k\le n-2$.
Denote $\pi_1(W_{n-1})/\S_2\pi_1(W_{n-1})$ as $\Lambda$ for simplicity.
As we saw in (\ref{eq3}), there is an injective map
\[\frac{\S_2\pi_1(W_{n-1})}{\S_3\pi_1(W_{n-1})}\hookrightarrow H_1(W_{n-1},\Z[\Lambda]S^{-1}).\]

Hence to show $\vp_{n-2}(\ol{\mu}_{n-2})$ vanishes, it is enough to prove that $\ol{\mu}_{n-2}=\ol{\eta}_{n-1}$ vanishes in $H_1(W_{n-1},\Z[\Lambda]S^{-1})$.

Note that $\ol{\eta}_{n-1}\in H_1(W_{n-1},\Z[\Lambda]S^{-1})$ lies in the image of inclusion-induced map from $H_1(M(\ol{K^{n-1}}),\Z[\Lambda]S^{-1})$, which is isomorphic to ${\Z[\Lambda]}/{\langle p(\b)\rangle}S^{-1}$.
Since $p(\b)$ lies in $S$, this module is trivial.
Hence $\ol{\eta}_{n-1}$ represents a trivial element in $H_1(W_{n-1},\Z[\Lambda]S^{-1})$ and this finishes the proof that $\ol{\mu}_{n-2}$ maps to $0$ by~$\vp_{n-2}$.

For $k<n-2$, suppose that $\ol{\mu}_{k+1}$ is mapped to $0$ by $\vp_{k+1}$.
Let $\langle\ol{\mu}_{k+1}\rangle$ be the subgroup of $\S_{n-k}\pi_1(W_{k+1})$ normally generated by~$\ol{\mu}_{k+1}$.
The curve $\eta_{k+1}$, considered to be in $M(\ol{J^{k+1}_1})\subset W_{k+1}$, represents an element in $\langle\ol{\mu}_{k+1}\rangle^{(1)}$.
Note that $\langle\ol{\mu}_{k+1}\rangle^{(1)}\subset S_{n-k+1}\pi_1(W_{k+1})$, and there is an inclusion-induced map $S_{n-k+1}\pi_1(W_{k+1})\to S_{n-k+1}\pi_1(W_{k})$.
Since $\ol{\mu}_k$ is isotopic to $\ol{\eta}_{k+1}$, $\ol{\mu}_k$ lies in $S_{n-k+1}\pi_1(W_{k})$, and by induction this finishes the proof.

\subsection{Proof of Theorem \ref{theorem:maintheorem3}}
In \cite{cha2010amenable}, it is shown that the classes of $(n)$-solvable knots with vanishing PTFA $L^2$-signature obstructions form a subgroup $\mathcal{V}_n$ of $\C$ such that $\F_{n.5}\le\mathcal{V}_n\le\F_n$.

By the Subsection~\ref{subsection:vanishing-PTFA-condition}, knots $J_i$ are $(n)$-solvable with vanishing PTFA $L^2$-signature obstructions.
Also by the Subsection~\ref{subsection:proof-of-theorem-A}, any nontrivial finite connected sum of $J_i$'s is not $(n.5)$-solvable.
Hence $J_i$'s form a $\Z_2^\infty$ subgroup in $\mathcal{V}_n/\F_{n.5}$.
On the other hand, the proof of Cochran, Harvey, and Leidy about non-$(n.5)$-solvability actually shows that their knots are not $(n)$-solvable with vanishing PTFA $L^2$-signature obstructions.
That is, they form a $\Z_2^\infty$ subgroup in $\F_n/\mathcal{V}_n$.

Hence the subgroup of $\F_n/\F_{n.5}$ generated by our knots $J_i$ and that of Cochran, Harvey, Leidy's 2-torsion knots have the trivial intersection.

\vskip 5mm

The remainder of this paper is dedicated to the statement and proof of Theorem~\ref{theorem:maintheorem2}.

\section{Cochran-Harvey-Leidy's $(h,\P)$-solvability of knots}\label{section:hp-solvability}

For $n\ge2$, let $\P$ be an $n$-tuple $(p_1(t),p_2(t),\dots,p_n(t))$ of Laurent polynomials over $\Z$.
In this section we briefly introduce $(h,\P)$-solvability of knots and a new filtration $\{\F_h^\P\}$ on the knot concordance group $\C$ defined and studied by Cochran, Harvey, and Leidy.
For more details, consult with \cite[Section~2]{cochran2011primary}.

\begin{definition}
Two Laurent polynomials $p(t)$ and $q(t)$ over $\Z$ are said to be \emph{strongly coprime}, and denoted by $\wt{(p,q)}=1$, if $p(t^m)$ and $q(t^n)$ are coprime in $\Z[t^{\pm1}]$ for all integers $m$ and~$n$.
\end{definition}

\begin{example}\label{example:strongly-coprime}
Consider the following set of polynomials
\[\{(kt-(k+1))((k+1)t-k)|k\in\Z^+\}.\]
It is shown that any two distinct polynomials in this set are strongly coprime \cite[Example~ 4.10]{cochran2011primary} .
\end{example}

\begin{definition}\label{defn:derived-series-localized}
Let $G$ be a group such that $G/G^{(1)}$ is isomorphic to $\Z$. 
The \emph{derived series localized at $\mathcal{P}=(p_1(t),p_2(t),\dots,p_n(t))$}, $\{G^{(i)}_\P\}^{n+1}_{i=0}$, is defined as follows:
\begin{enumerate}
\item $G^{(0)}_\mathcal{P}= G$ and $G^{(1)}_\mathcal{P}=G^{(1)}$.
\item
Let $\mu$ be a generator of $G/G^{(1)}_\P=\Z$ and $S_{p_n}^*$ a multiplicative subset of $\Z[G/G^{(1)}_\P]$ 
$$\{q_1(\mu^{\pm1})\dots q_r(\mu^{\pm1})|q_j(t)\in\Z[t^{\pm1}], (p_n,q_j)=1, q_j(1)\neq 0\}.$$
Then $G^{(2)}_\mathcal{P}$ is defined as the kernel of the map
\[ G^{(1)}_\mathcal{P}\to\frac{G^{(1)}_\mathcal{P}}{[G^{(1)}_\mathcal{P},G^{(1)}_\mathcal{P}]}\otimes_{\Z[G/G^{(1)}_\mathcal{P}]}\Z[G/G^{(1)}_\mathcal{P}](S^{*}_{p_n})^{-1}.\]

\item
For any $2\leq i\leq n$, suppose that $G_\P^{(i)}$ is defined. Then let
\[S_i=\{q_1(a_1)\dots q_r(a_r)|q_j(t)\in\Z[t^{\pm1}], \wt{(p_{n-i+1},q_j)}=1,q_j(1)\neq0,a_j\in G^{(i-1)}_\mathcal{P}/G^{(i)}_\mathcal{P}\}.\]
By Proposition~\ref{proposition:right-divisor-set}, $S_i$ is a right divisor set of $\Z[G/G^{(i)}_\mathcal{P}]$.
Let $G^{(i+1)}_\mathcal{P}$ be the kernel of the~map
\[G^{(i)}_\mathcal{P}\to\frac{G^{(i)}_\mathcal{P}}{[G^{(i)}_\mathcal{P},G^{(i)}_\mathcal{P}]}\otimes_{\Z[G/G^{(i)}_\mathcal{P}]}\Z[G/G^{(i)}_\mathcal{P}]S_{i}^{-1}.\]

\end{enumerate}
\end{definition}

Now we are ready to define $(h,\P)$-solvability of knots for any half integer $0\leq h\le n+1$.

\begin{definition}\cite[Definition~2.3]{cochran2011primary}
For an integer $0\leq k \leq n+1$, a knot $K$ is called \emph{$(k,\P)$-solvable} if its zero framed surgery manifold $M(K)$ bounds a compact smooth $4$-manifold $W$ such that
\begin{enumerate}
\item  the inclusion-induced map $H_1(M(K),\Z)\to H_1(W,\Z)$ is an isomorphism,
\item $H_2(W,\Z)$ has a basis consisting of connected compact oriented surfaces, $\{L_i,D_i\}$, embedded in $W$ with trivial normal bundles, wherein the surfaces are pairwise disjoint except that, for each $i$, $L_i$ intersects $D_i$ transversely once with positive sign, and,
\item for each $i$, $\pi_1(L_i)\subset \pi_1(W)^{(k)}_\P$ and $\pi_1(D_i)\subset \pi_1(W)^{(k)}_\P$.
\end{enumerate}

A knot $K$ is $(k.5,\P)$-solvable if in addition, for each $i$, $\pi_1(L_i)\subset\pi_1(W)^{(k+1)}_\P$.

If a knot $K$ is $(h,\P)$-solvable via a $4$-manifold $W$, then $W$ is called a \emph{$(h,\P)$-solution for~$K$}.
\end{definition}

The only difference between the ordinary $(h)$-solvability and $(h,\P)$-solvability is on the use of distinct normal series at the property 3 and 4. 
The former uses the ordinary derived series and the latter uses the derived series localized at $\P$ \cite[Proposition~2.5]{cochran2011primary}.

Cochran, Harvey, and Leidy proved that a knot concordant to an $(h,\P)$-solvable knot is also $(h,\P)$-solvable, and the subset $\F^\P_h\subset\mathcal{C}$ of classes of $(h,\P)$-solvable knots is a subgroup.
Hence they form a filtration on~$\mathcal{C}$:
\[\F^\P_{n+1}\leq\F^\P_{n.5}\leq\dots\leq\F^\P_1\leq\F^\P_{0.5}\leq\F^\P_0\leq\mathcal{C}.\]



Since $G^{(i)}_\P$ contains $G^{(i)}$ for any group $G$ and $i$, $\F_h$ is contained in $\F^\P_h$.
This provides the following quotient map:
\[\phi_\mathcal{P}\colon\frac{\F_n}{\F_{n.5}}\to\frac{\F_n}{\F_{n.5}^\mathcal{P}\cap\F_n}.\]

In this setting, the main result of \cite{cochran20112-torsion} can be summarized in the following form:
\begin{theorem}[Theorem 5.5 of \cite{cochran20112-torsion}]\label{theorem:cochran-harvey-leidy-2-torsion}
Let $\P=(p_1(t),p_2(t),\dots,p_n(t))$ be an $n$-tuple such~that
\begin{enumerate}
\item for $k<n$, $p_k(t)=\delta(t)^2$ for some nonzero nonunit $\delta(t)\in\Z[t^{\pm1}]$ with $\delta(1)=\pm1$ and $\delta(t)=\delta(t^{-1})$ up to the multiplication by $t^i$,
\item $p_n(t)=m^2t^2-(2m^2+1)t+m^2$ for some nonzero integer $m$.
\end{enumerate}
Then there is a negatively amphichiral $(n)$-solvable knot which
\begin{enumerate}
\item is mapped injectively by $\phi_\P$, but
\item is mapped trivially by $\phi_\mathcal{Q}$ for every $n$-tuple $\mathcal{Q}$ strongly coprime to~$\P$.
\end{enumerate}
\end{theorem}
Here, two $n$-tuples $\P=(p_1(t),p_2(t),\dots,p_n(t))$ and $\mathcal{Q}=(q_1(t),q_2(t),\dots,q_n(t))$ are said to be \emph{strongly coprime} if either $(p_n,q_n)=1$ or $\widetilde{(p_k,q_k)}=1$ for some $k<n$.

\begin{corollary}\label{thm:chl-revised}
Let $\P=(0,p_2(t),\dots,p_n(t))$ be an $n$-tuple such that
\begin{enumerate}
\item for $1<k<n$, $p_k(t)=\delta(t)^2$ for some nonzero nonunit $\delta(t)\in\Z[t^{\pm1}]$ with $\delta(1)=\pm1$ and $\delta(t)=\delta(t^{-1})$ up to the multiplication by $t^i$,
\item $p_n(t)=m^2t^2-(2m^2+1)t+m^2$ for some nonzero integer $m$.
\end{enumerate}
Then there is a  $\Z_2^\infty$ subgroup of $\F_n/\F_{n.5}$ which
\begin{enumerate}
\item is mapped injectively by $\phi_\P$, but
\item is mapped trivially by $\phi_\mathcal{Q}$ for every $n$-tuple $\mathcal{Q}$ which is strongly coprime to~$\P$.
\end{enumerate}
\begin{proof}
Let $\{f_1(t),f_2(t),\dots\}$ be an infinite set of pairwise strongly coprime Laurent polynomials of the form $\delta(t)^2$ for some $\delta(t)\in\Z[t^{\pm1}]$ with $\delta(1)=\pm1$ and $\delta(t)=\delta(t^{-1})$ up to the multiplication by $t^i$.
The set of polynomials $\{(kt-(k + 1))((k + 1)t-k)|k\in\Z\}$ in Example~\ref{example:strongly-coprime} is an instance.
Let $\P_k=(f_k(t),p_2(t),\dots,p_n(t))$.
By the choice of $f_k(t)$, $\P_1, \P_2, \dots$ are pairwisely strongly coprime.
By Theorem~\ref{theorem:cochran-harvey-leidy-2-torsion}, there is a negatively amphichiral $(n)$-solvable knot, say $L_k$, which is not $(n.5, \P_k)$-solvable but $(n.5, \P_m)$-solvable for all $m\neq k$.
This implies that the knots $L_k$ are linearly independent modulo 2 in $\F_n/(\F_{n.5}^\P\cap \F_n)$.
\end{proof}
\end{corollary}

As an obstruction for a knot being non-$(n.5,\P)$-solvable, they refined Theorem~4.2 of \cite{cochran2003knot}:
\begin{theorem}[Theorem 5.2  of \cite{cochran2011primary}]\label{thm:chl}
Let $K$ be an $(n.5,\P)$-solvable knot with a solution $V$. 
For a PTFA group $\Gamma$, let $\phi\colon\pi_1(M(K))\to\Gamma$ be a homomorphism which factors through $\pi_1(V)/\pi_1(V)^{(n+1)}_\P$.
Then $S(V,\phi)$ vanishes.
\end{theorem}

This vanishing theorem can be extended to amenable group homomorphisms.

\begin{theorem}\label{thm:refinedAST}
Let $K$ be an $(n.5,\P)$-solvable knot with a solution $V$. For an amenable group $\Gamma$ in Strebel's class $D(R)$ for a ring $R$, let $\phi\colon\pi_1(M(K))\to\Gamma$ be a group homomorphism which factors through $\pi_1(V)/\pi_1(V)^{(n+1)}_\P$.
Then $S(V,\phi)$ vanishes.
\end{theorem}
Since the proof is identical to the proof of Theorem 1.3 of \cite{cha2010amenable}, we omit the proof.

\section{Proof of Theorem~\ref{theorem:maintheorem2}}\label{section:proof-of-theorem-B}
In this section we prove the following more detailed version, from which Theorem~\ref{theorem:maintheorem2} follows immediately:

\begin{theorem}\label{theorem:pricisemaintheorem2}
Let $\P=(0,p_2(t),\dots,p_n(t))$ be an $n$-tuple of Laurent polynomials over $\Z$ such that

\begin{enumerate}
\item for $1<k<n$, $p_k(t)=\delta(t)^2$ for some nonzero nonunit $\delta(t)\in\Z[t^{\pm1}]$ with $\delta(1)=\pm1$ and $\delta(t)=\delta(t^{-1})$ up to the multiplication by~$t^i$,
\item $p_n(t)=m^2t^2-(2m^2+1)t+m^2$ for some nonzero integer~$m$.
\end{enumerate}
Then there is a $\Z_2^\infty$ subgroup of $\F_n/\F_{n.5}$ which
\begin{enumerate}
\item is mapped injectively by $\phi_\P$, but
\item is mapped trivially by $\phi_\mathcal{Q}$ for every $n$-tuple $\mathcal{Q}$ which is strongly coprime to~$\P$.
\end{enumerate}

Also, the subgroup of $\F_n/(\F_{n.5}^\P\cap\F_n)$ generated by these knots has trivial intersection with the subgroup generated by the Cochran, Harvey and Leidy's knots (Corollary~\ref{thm:chl-revised}).
\end{theorem}

We remark that the overall outline of the proof is parallel to that of the proof for Theorem~\ref{theorem:maintheorem1}.

\subsection{Proof of Theorem \ref{theorem:pricisemaintheorem2} modulo infection axis analysis}

Throughout this section, we fix an $n$-tuple $\P$ which satisfies the conditions in Theorem~\ref{theorem:pricisemaintheorem2}.
We construct $(n)$-solvable knots $J_1,J_2,\dots$ using the iterated infection construction as in Section~\ref{section:construction}, with the following choice of $K$, $K^k$, and~$J^0_i$'s:

\begin{enumerate}
\item[$K$ :] Let $K$ be the connected sum of two copies of $\mathbb{E}_m$ (see Figure~ \ref{figure:E_m} and \ref{figure:K}).
We take two axes $\a$ and $\b$ as in Figure~\ref{figure:K}.
\item[$K^k$:]
For any $q(t)\in\Z[t^{\pm1}]$ such that $q(1)=\pm1$ and $q(t^{-1})=q(t)$ up to the multiplication by $t^i$, there is a slice knot whose classical Alexander module is a cyclic $\Z[t^{\pm1}]$-module of the form $\Z[t^{\pm1}]/\langle q(t)^2\rangle$ (see \cite[Theorem~5.18]{cha2007the}).
For $1<k<n$, we choose a slice knot whose Alexander module is $\Z[t^{\pm1}]/\langle p_k(t)\rangle$ as $K^k$.

Let $\delta(t)\in\Z[t^{\pm1}]$ be a nonzero Laurent polynomial such that $\delta(1)=\pm1$, $\delta(t^{-1})=\delta(t)$ up to the multiplication by $t^i$, and $\ol{\delta}(t)\in\Z_d[t^{\pm1}]$ is nonunit for any prime $d$.
By Theorem 12 and 13 of \cite{higman1940the}, polynomial $t^2-t+1$ for example works.
Let $K^1$ be a slice knot whose Alexander module is $\Z[t^{\pm1}]/\langle \delta(t)^2\rangle$.

Choose a closed curve $\eta_k$ in $S^3-K^k$ which is unknotted in $S^3$ and represents a generator of the Alexander module of $K^k$.
Note that we can always find such a curve (see the proof of \cite[Theorem 3.8]{cochran2007knot}).

\item[$J^0_i$:] By Theorem~\ref{thm:cheeger-gromov-bound}, there is a constant $C$ which always is greater than
\[\bigg|\rho(M(K),\phi_K)+\sum_{k=1}^{n-1} \rho(M(K^k),\phi_k)+\sum_{k=1}^{n-1} \rho(M(\ol{K^k}),\ol{\phi}_k)\bigg|, \]
independent of the choice of homomorphisms $\phi_K$, $\phi_k$ and $\ol{\phi}_k$, $k=1,\dots,n-1$ on $M(K)$, $M(K^k)$, and $M(\ol{K^k})$, respectively.
Apply Proposition~\ref{prop:knot-of-cha} with this constant and construct knots~$J^0_i$.
\end{enumerate}

To show that any nontrivial finite connected sum of knots $J_i$ cannot be $(n.5,\P)$-solvable, we proceed similarly to the proof of Theorem~\ref{theorem:maintheorem1} in Section~\ref{section:proof-of-theorem-A}: 
first assume that $J\equiv J_1\#\dots\# J_r$ is $(n.5,\P)$-solvable, and then construct a 4-manifold $W_0$ in Figure~\ref{figure:W_k} and a group homomorphism $\vp$ on $\pi_1(W_0)$, via certain normal series which we will denote as $\T_k\pi_1(W_0)$, $k=0,\dots,n+1$.
Then two evaluations of $S(W_0,\vp)$, which produce different values, give a contradiction, showing that $J$ cannot be $(n.5,\P)$-solvable.

Now we define a subgroup $\T_k\pi_1(W_i)$ of $\pi_1(W_i)$.
For simplicity, denote $\pi_1(W_i)$ as $\pi$.
For $k=0,1,2$, let $\T_k\pi$ be $\pi^{(k)}_\P$, defined in Definition~\ref{defn:derived-series-localized}.

Let $T_2$ be the multiplicative subset of $\Q[\T_1\pi/\T_2\pi]\subset\Q[\pi/\T_2\pi]$ generated by
\[\{q(a)|q(t)\in\Z[t^{\pm1}], \widetilde{(q,p_{n-1})}=1,q(1)\neq0,a\in\T_1\pi/\T_2\pi    \} \cup \{p_{n-1}(\mu^i\b\mu^{-i})|i\in\Z\},\]
where $\mu$ is the meridian of $K$.
By Proposition~\ref{proposition:right-divisor-set}, $T_2$ is a right divisor set of $\Q[\pi/\T_2\pi]$.
Let
\[\T_3\pi=\ker\left(\T_2\pi\to \frac{\T_2\pi}{[\T_2\pi,\T_2\pi]}\otimes_{\Z[\pi/\T_2\pi]}\Q[\pi/\T_2\pi]T_2^{-1}     \right).\]

For $2<k<n$, let $T_k$ be the multiplicative subset of $\Q[\T_{k-1}\pi/\T_k\pi]$ generated by
\[\{q(a)|q(t)\in\Z[t^{\pm1}], \widetilde{(q,p_{n+1-k})}=1,q(1)\neq0,a\in\T_{k-1}\pi/\T_k\pi    \},\]
which is a right divisor set by Proposition~\ref{proposition:right-divisor-set}, and let
\[\T_{k+1}\pi=\ker\left(\T_k\pi\to \frac{\T_k\pi}{[\T_k\pi,\T_k\pi]}\otimes_{\Z[\pi/\T_k\pi]}\Q[\pi/\T_k\pi]T_k^{-1}     \right).\]

Finally, let
\[\T_{n+1}\pi=\ker\left(\T_n\pi\to \frac{\T_n\pi}{[\T_n\pi,\T_n\pi]}\otimes_{\Z}\Z_d     \right),\]
where $d=d_1$.

Let $G=\pi/\T_{n+1}\pi$ and $\vp\colon\pi_1(W_0)\to G$ be the quotient map.
Note that $G$ is amenable and lies in the Strebel's class $D(\Z_d)$ by Lemma~\ref{lemma:amenable_group}.
We calculate the $L^2$-signature defect $S(W_0,\vp)$.

First, calculate $S(W_0,\vp)$ by using Novikov additivity theorem (Theorem~\ref{thm:Novikov-additivity}).
The calculation is parallel to \emph{First method} in Section~\ref{section:proof-of-theorem-A}, except the use of obstruction theorem for $(n.5,\P)$-solvability (Theorem~\ref{thm:refinedAST}), instead of for $(n.5)$-solvability (Theorem~\ref{thm:homology-cobordism-invariance-on-amenable})
We omit the detail.
The calculation gives that $S(W_0,\vp)$ is zero.

Next, we calculate $S(W_0,\vp)$ by using the fact that $S(W_0,\vp)=\rho(\d W_0,\vp)$.
We follow the exactly same argument as \emph{Second method} in the proof of Theorem~\ref{theorem:maintheorem1}, except the following lemma in place of Lemma~\ref{Lemma:nontriviality}:

\begin{lemma}\label{Lemma:nontriviality2}
For $k=0,1,\dots,n-1$, the homomorphism
\[\vp_k\colon \pi_1(W_k)\to \frac{\pi_1(W_k)}{\T_{n-k+1}\pi_1(W_k)}\]
sends the meridian $\mu_k$ of $J^k$ and $\ol{\mu}_k$ of $\ol{J^k}$ into the subgroup $\T_{n-k}\,\pi_1(W_k)/\T_{n-k+1}\,\pi_1(W_k)$.
Also, $\vp_k(\mu_k)$ is nontrivial for any $k=0,\dots,n-1$, while $\vp_k(\ol{\mu}_k)$ is nontrivial for $k=n-1$ and trivial for other $k$.
\end{lemma}
This provides that $S(W_0,\vp)$ is nonzero.
Two contradictory calculations of $S(W_0,\vp)$ finish the nontriviality result of Theorem~\ref{theorem:pricisemaintheorem2}.

In Theorem~5.3 of \cite{cochran20112-torsion} it is proved that, for $\mathcal{Q}$ strongly coprime to $\P$, $J_i$ vanishes in $\F_n/(\F_{n.5}^\mathcal{Q}\cap\F_n)$.
This finishes the proof of Theorem~\ref{theorem:pricisemaintheorem2}.

\subsection{Nontriviality of the infection axes}

In this subsection we prove Lemma~\ref{Lemma:nontriviality2}.
As we saw in the proof of Lemma~\ref{Lemma:nontriviality}, $\vp_k(\mu_k)$ and $\vp_k(\ol{\mu}_k)$ lie in $\pi_1(W_k)^{(n-k)}$, hence in $\T_{n-k}\pi_1(W_k)$, for all $k$.
Also, a Mayer-Vietoris sequence argument shows that $\mu_n$ does not vanish in $$\pi_1(W_n)/\T_1\pi_1(W_n)= H_1(W_n)=\Z.$$

We show that neither $\vp_{n-1}(\mu_{n-1})$ nor $\vp_{n-1}(\ol{\mu}_{n-1})$ are zero.
Since $p_n(t)=m^2t^2-(2m^2+1)t+m^2$ is the Alexander polynomial of $\mathbb{E}_m$ (e.g., see, \cite[Example~4.10]{cochran2011primary}) and $\Q[t^{\pm1}](S^*_{p_n})^{-1}$ is a flat $\Q[t^{\pm1}]$-module, $H_1(M(\mathbb{E}_m),\Q[t^{\pm1}](S^*_{p_n})^{-1})$ is isomorphic~to 
\[\frac{\Q[t^{\pm1}]}{\langle p_n(t)\rangle}\otimes_{\Z[t^{\pm1}]} \Q[t^{\pm1}](S^*_{p_n})^{-1}.\]
Since ${\Q[t^{\pm1}]}/{\langle p_n(t)\rangle}$ is $S^*_{p_n}$-torsion free (see Proposition~4.13 of \cite{cochran2011primary}), the map
\begin{equation*}
\Psi\colon H_1(M(\mathbb{E}_m),\Q[t^{\pm1}])\to H_1(M(\mathbb{E}_m),\Q[t^{\pm1}](S^*_{p_n})^{-1})\end{equation*}
induced from the inclusion $\Q[t^{\pm1}]\hookrightarrow\Q[t^{\pm1}](S^*_{p_n})^{-1}$ is injective.

Note that the localized Alexander module of $K=\mathbb{E}_m\#\mathbb{E}_m$ splits into the direct sum of two copies of localized Alexander modules (localized with the coefficient system) of~$\mathbb{E}_m$.

Let $\Phi$ be the following composition:
\[
\begin{diagram}
\node{H_1(M(\mathbb{E}_m),\Q[t^{\pm1}](S^*_{p_n})^{-1})} \arrow{e}
\node{H_1(M(K),\Q[t^{\pm1}](S^*_{p_n})^{-1})} \arrow{e,t}{\iota_*}
\node{H_1(W_{n-1},\Q[t^{\pm1}](S^*_{p_n})^{-1}),}
\end{diagram}
\]
where the homologies are twisted by abelianization maps.
The first map is induced from the left copy of $\mathbb{E}_m$ of $K=\mathbb{E}_m \#\mathbb{E}_m$ (see Figure~\ref{figure:K}) and the second map is induced from the inclusion $\iota\colon M(K)\hookrightarrow W_{n-1}$.

Similar to the classical Alexander module case, the Blanchfield form on $H_1(M(K),\Q[t^{\pm1}](S^*_{p_n})^{-1})$ splits as the sum of two copies of Blanchfield form on $H_1(M(\mathbb{E}_m),\Q[t^{\pm1}](S^*_{p_n})^{-1})$.
Note that $H_1(M(\mathbb{E}_m),\Q[t^{\pm1}](S^*_{p_n})^{-1})$ is a simple $\Q[t^{\pm1}](S^*_{p_n})^{-1}$-module.
Using this fact, the same argument with that of Lemma~\ref{lemma:injective} proves $\Phi$ is injective.

As in Section~\ref{section:proof-of-theorem-A}, the injectivity of $\Phi\circ\Psi$ implies that both $\mu_{n-1}$ and $\ol{\mu}_{n-1}$ are nontrivial in $\pi_1(W_{n-1})/\T_2\pi_1(W_{n-1})$.

For $k\leq n-2$, suppose that the homomorphism
\[\vp_{k+1}\colon\pi_1(W_{k+1})\to \frac{\pi_1(W_{k+1})}{\T_{n-k}\pi_1(W_{k+1})}\]
sends $\mu_{k+1}$ into $\T_{n-k-1}\pi_1(W_{k+1})/\T_{n-k}\pi_1(W_{k+1})$ nontrivially.

Denote $\T_{n-k}\pi_1(W_{k+1})$ by $\Lambda$ temporarily.
Let $\mathcal{R}$ be $\Q[\pi_1(W_{k+1})/\Lambda]T_{n-k}^{-1}$ for $k>0$ and $\Z_d[\pi_1(W_{k+1})/\Lambda]$ for $k=0$.
Recall that $\T_{n-k+1}\pi_1(W_{k+1})$ is the kernel of the map
\[\Lambda=\T_{n-k}\pi_1(W_{k+1}) \to \frac{\Lambda}{[\Lambda,\Lambda]}\otimes_{\Z[\pi_1(W_{k+1})/\Lambda]}\mathcal{R}\hookrightarrow H_1(W_{k+1},\mathcal{R}).  \]

So we have a natural injective map
\[\frac{\T_{n-k}\pi_1(W_{k+1})}{\T_{n-k+1}\pi_1(W_{k+1})}\hookrightarrow H_1(W_{k+1},\mathcal{R}).\]

By Lemma \ref{lemma:isomorphism}, we have an isomorphism
\[\frac{\T_{n-k}\pi_1(W_{k+1})}{\T_{n-k+1}\pi_1(W_{k+1})}\cong \frac{\T_{n-k}\pi_1(W_k)}{\T_{n-k+1}\pi_1(W_k)}.\]

Hence now we have a map from $\T_{n-k}\pi_1(W_k)/\T_{n-k+1}\pi_1(W_k)$ to $H_1(W_{k+1},\mathcal{R})$.

Now, suppose $\mu_k$ maps trivially by $\vp_k$, that is, $\mu_k$ is in $\T_{n-k+1}\pi_1(W_k)$.
Since $\mu_k=\eta_{k+1}$ in $\pi_1(W_{k})$, $\eta_{k+1}$, isotoped into $W_{k+1}$, vanishes in $H_1(W_{k+1},\mathcal{R})$.
Hence $\eta_{k+1}\subset M(J_1^{k+1})$ lies in the kernel of the inclusion-induced~map
\[H_1(M(J_1^{k+1}),\mathcal{R})\to H_1(W_{k+1},\mathcal{R})  \]

The inductive assumption enables us to apply Theorem~\ref{theorem:cochran-harvey-leidy-theorem}, so $\Bl([\eta_{k+1}],[\eta_{k+1}])$ for $\Bl$ the Blanchfield form on $H_1(M(J_1^{k+1}),\mathcal{R})$ vanishes.

Note that $H_1(M(J^{k+1}_1),\mathcal{R})$ is isomorphic to $$H_1(M(K^{k+1}),\Z[t^{\pm1}])\otimes_{\Z[t^{\pm1}]}\mathcal{R}=\frac{\Z[t^{\pm1}]}{\langle p_{k+1}(t)\rangle}\otimes_{\Z[t^{\pm1}]}\mathcal{R}.$$
Since $\eta_{k+1}$ generates $H_1(M(K^{k+1}),\Z[t^{\pm1}])$, $\eta_{k+1}$ also generates $H_1(M(J^{k+1}_1),\mathcal{R})$ as an $\mathcal{R}$-module.
Then Theorem~\ref{lemma:nonsingularity} assures that $$\frac{\Z[t^{\pm1}]}{\langle p_{k+1}(t)\rangle}\otimes_{\Z[t^{\pm1}]}\mathcal{R}$$ is a zero module.


The non-unital property of the Alexander polynomial $p_1(t)$ of $K^1$ assures that this module is nontrivial when $k=0$.
Also the following theorem shows that $H_1(M(J^{k+1}_1),\R)$ is nontrivial, possibly except the case when $k=n-2$.
\begin{theorem}\cite[Theorem~4.12]{cochran2011primary}
Suppose $A$ is a normal subgroup of $\Gamma$ where $A$ is a torsion-free abelian group and $\Q\Gamma$ is a right Ore domain. Suppose $p(t)\in\Q[t^{\pm1}]$ is non-zero. Then for any $a_i\in A$,
\[\frac{\Q\Gamma}{\langle p(a_1)\dots p(a_r)\rangle}\textrm{ is }S_p\textrm{-torsion-free,}	\]
while for any $q(t)\in\Q[t^{\pm1}]$ with $q(1)\neq0$ and $\widetilde{(p(t),q(t))}=1$
\[\frac{\Q\Gamma}{\langle q(a)\rangle}\textrm{ is }S_p\textrm{-torsion,}	\]
for any $a\in A$.
\end{theorem}

So it remains to check that $H_1(M(J_1^{n-1}),\mathcal{R})$ is not a trivial module
Suppose not.
Then
\[\frac{\Z[t^{\pm1}]}{\langle p_{n-1}(t)\rangle}\otimes_{\Z[t^{\pm1}]}\mathcal{R}\cong \frac{\Q[\Lambda/\T_2\Lambda]}{\langle p_{n-1}(\a)\rangle}T_2^{-1}=0,\]
where $\pi_1(W_{n-1})$ is denoted by $\Lambda$ for simplicity.

This implies that the unity $\ol{1\cdot e}$ represents 0 in this module where $e$ is the identity element of the group $\Lambda/\T_2\Lambda$.
Then there is $s\in T_2$ and $f\in \Q[\Lambda/\T_2\Lambda]$ such that $1\cdot s=p_{n-1}(\a)\cdot f$ in $\Q[\Lambda/\T_2\Lambda]$.
Since $p_{n-1}(\a)$ and $s$ are in $\Z[\T_1\Lambda/\T_2\Lambda]$, $f$ also lies in $\Z[\T_1\Lambda/\T_2\Lambda]$.
The following proposition \cite[Proposition~4.5]{cochran2011primary} shows that $s$ is a product of elements of the form $p_{n-1}(\mu^i\b\mu^{-i})$.
\begin{proposition}
Suppose $p(t),q(t)\in\Q[t^{\pm1}]$ are non-zero.
Then $p$ and $q$ are strongly coprime if and only if, for any finitely generated free abelian group $F$ and any nontrivial $a,b\in F$, $p(a)$ is relatively prime to $q(b)$ in $\Q F$.
\end{proposition}

Now the same argument in the last part of the proof of Lemma~\ref{Lemma:nontriviality} shows that this cannot be true. 
This finishes the proof that $\vp(\mu_k)$ is nontrivial.

That the image of $\ol{\mu}_k$ under $\vp_k$ is trivial for $k<n-1$ can be shown in the similar manner as in the proof of Lemma~\ref{Lemma:nontriviality} so we omit the proof.

\bibliographystyle{amsalpha}
\renewcommand{\MR}[1]{}
\bibliography{research}

\end{document}